\documentclass[review]{elsarticle}
\usepackage{lineno,hyperref}
\usepackage{amssymb}
\usepackage{bm}
\usepackage{mathrsfs}
\usepackage[english]{babel}
\usepackage{amsmath, cases,amsfonts,mathrsfs}
\usepackage{epsfig, multirow, subfigure, txfonts}
\usepackage[small]{caption2}
\usepackage{color,hyperref}
\usepackage[thmmarks]{ntheorem}
{
\theoremstyle{nonumberplain}
\theoremheaderfont{\bfseries}
\theorembodyfont{\normalfont}
\theoremsymbol{\mbox{$\Box$}}
\newtheorem{proof}{Proof.}
}
\def\bc{\begin{center}}
\def\ec{\end{center}}

\def\bel{\begin{equation}}
\def\enl{\end{equation}}
\def\be{\begin{eqnarray*}}
\def\en{\end{eqnarray*}}

\newtheorem{Th}{Theorem}[section]
 \newtheorem{Lem}{Lemma}[section]

\newtheorem{Pro}{Proposition}[section]
\newtheorem{Def}{Definition}[section]
\newtheorem{Rem}{Remark}[section]

\def\R{{\mathbb{R}}}

\def\H{{\mathbb{H}}}

\def\F{{\mathcal {F}}}

\modulolinenumbers[10]
\begin{document}

\begin{frontmatter}

\title{Prolate Spheroidal Wave Functions Associated with the Quaternionic Fourier Transform \tnoteref{mytitlenote}}
\author[firstaddress]{Cuiming~Zou}
%\ead{zoucuiming2006@163.com}

\author[mymainaddress1]{Kit Ian Kou\corref{mycorrespondingauthor}}
\cortext[mycorrespondingauthor]{Corresponding author}
%\ead{kikou@umac.mo}

\author[mymainaddress2]{Joao Morais}
%\ead{joao.morais@itam.mx}

\address[firstaddress]{Department of Mathematics, Faculty of Science and Technology,
University of Macau, Taipa, Macao, China. Email: zoucuiming2006@163.com}
\address[mymainaddress1]{Department of Mathematics, Faculty of Science and Technology,
University of Macau, Taipa, Macao, China. Email:kikou@umac.mo}
\address[mymainaddress2]{J. Morais is with the Departamento de Matem\'aticas,
Instituto Tecnol\'ogico Aut\'onomo de M\'exico
R\'io Hondo \#1, Col. Progreso Tizap\'an, M\'exico, DF 01080, M\'exico. Email:joao.morais@itam.mx}

\begin{abstract}
One of the fundamental problems in communications is finding the energy distribution of signals in time and frequency domains.
It should, therefore, be of great interest to find the most energy concentration hypercomplex signal.
The present paper finds a new kind of hypercomplex signals whose energy concentration
is maximal in both time and frequency under quaternionic Fourier transform.
The new signals are a generalization of the prolate spheroidal wave functions (also known as Slepian functions)
to quaternionic space, which are called quaternionic prolate spheroidal wave functions.
The purpose of this paper is to present the definition and properties of the quaternionic prolate spheroidal wave functions and
to show that they can reach the extreme case in energy concentration problem both from the theoretical and experimental description.
In particular, these functions are shown as an effective method for bandlimited signals extrapolation problem.
\end{abstract}

\begin{keyword}
Quaternionic analysis \sep quaternionic Fourier transform \sep hypercomplex signal \sep
energy concentration problem \sep quaternionic prolate spheroidal wave functions \sep bandlimited extrapolation
\MSC[2016] 00-01\sep  99-00
\end{keyword}
\end{frontmatter}

\linenumbers
%-----------------------------------------------------------------------------------------------------------------------------
\section{Introduction}
\label{S1}
The energy distribution problem \cite{R1968,TP1987}
is one of the fundamental problems in communication engineering.
It aims at finding the energy distribution of signals  in time and frequency domains.
In particular, researchers take particular note of finding the signals
with maximum energy concentration in both the time and frequency domains simultaneously.
In the early 1960's, D. Slepian, H. Landau and H. Pollak \cite{SP1961, LP1961,LP1962,S1964}
have found that the prolate spheroidal wave functions (PSWFs) are the most optimal energy concentration functions
in a Euclidean space of finite dimension. The PSWFs, also known as Slepian functions,
are bandlimited and exhibit interesting orthogonality relations.
They are normalized versions of the solutions to the Helmholtz wave equation in prolate spheroidal coordinates.
From then on, theory and numerical applications of these functions have been developed rapidly.
They are often regarded as somewhat mysterious, with no explicit or standard representation
in terms of elementary functions and too difficult to compute numerically.

The one-dimensional PSWFs have so far mainly been developed in two different directions.
One is that fast and highly accurate methods have been deeply developed for the approximation of the PSWFs
and their eigenvalues \cite{LW1980,MC2004,KM2008,S1967,WS2005}.
From a computational point of view, the PSWFs provide a natural and efficient tool for
computing with bandlimited functions defined on an interval \cite{OR2012,K2010}.
They are preferable to classical polynomial bases (such as Legendre and Chebychev polynomials).
On the other hand, analytical properties of the PSWFs have been proposed of the more general context of
a variety of function spaces such as hypercomplex and Cliffordian spaces and
under different integral transforms \cite{MKZ2013,MK2014,Z2007,WS2004}.
The frequency domain has indeed been considered not only under the Fourier transform, but also under
more general ones such as the fractional Fourier transform \cite{Z2014, PD2005},
linear canonical transform (LCT) \cite{ZRMT2010, ZWSW2014} and  so forth.
Over the past few years the PSWFs have come to play a very active part in some problems
arising from applications to physical sciences and engineering, such as wave scattering,
signal processing, and antenna theory. These applications have stimulated a surge of new ideas and methods,
both theoretical and applied, and have reawakened an interest in approximation theory,
potential theory and the theory of partial differential equations.

In the present paper, we discuss the energy concentration problem of bandlimited quaternionic signals under
the quaternionic Fourier transform, which is a generalization of the Fourier transform to
quaternionic signals \cite{BHHA2008,E1993,H2007,HM2008}.
Quaternions were first applied to the Fourier transform within the general context of
solving problems in the nuclear magnetic resonance imaging by R. R. Ernst in the 1980s \cite{EBW1987}.
In the quaternionic language the Hamiltonian multiplication rules \cite{S1979}
make the underlying elements of having a non-commutative property
leading to the main difficulty for the analysis of some basic properties of those elements
in certain spaces.  To the best of our knowledge, there are no previous works considering the energy concentration problem
in a non-commutative structure as in the case of the quaternionic (or, more general, the Clifford) algebra.
This understanding can be the basis for more generalizations.
For this reason, the purpose of the present paper is to develop the energy extremal properties
between the time and frequency domains involving quaternionic signals.
We find that the quaternionic prolate spheroidal wave functions
(QPSWFs) preserve the high energy concentration in both the time and frequency domains
under the quaternionic Fourier transform.
Although we concentrate mainly on Fourier bandlimited signals,
it however might also be applied for LCT bandlimited cases as is shown
in one of our preceding papers \cite{MKZ2013}.

The body of the present paper will cover the following sequence of topics:
In Section \ref{S2}, we collect some basic concepts of quaternionic analysis and quaternionic Fourier transform to be used throughout the paper.
Section \ref{S3} introduces the QPSWFs.
The QPSWFs are ideally suited to study certain questions regarding the relationship
between quaternionic signals and their Fourier transforms.
We prove that the QPSWFs are orthogonal and complete on both
the square integrable space of finite interval and the two-dimensional
Paley-Wiener space of bandlimited quaternionic signals.
Section \ref{S4} describes the time-limited and bandlimited quaternionic signals and
their corresponding properties using the QPSWFs.
In Section \ref{S5}, we present the energy extremal properties of the QPSWFs in the time and frequency domains.
In particular, if a finite energy signal is given, the possible proportions of its energy in a finite time-domain and
a finite frequency-domain are found, as well as the signals which do the best job of simultaneous time and frequency
concentration. We find that the QPSWFs can reach the extreme case of the energy relationships.
In Section \ref{S6}, the QPSWFs are used in the  bandlimited extrapolation problem and achieved good results.
Section \ref{S7}  concludes the paper.

%%%%%%%%%%%%%%%%%%%%%%%%%%%section2%%%%%%%%%%%%%%%%%%%%%%%%%%%%%%%%%%%%%%%%%%%%%%
\section{The quaternionic Fourier transform (QFT)}
\label{S2}
\subsection{Quaternion Algebra}
Quaternion algebra was discovered by W.R. Hamilton in 1843.
It is a kind of hypercomplex numbers related to the rotations in three-dimensional space,
which likes complex numbers used to represent rotations in the two-dimensional plane.
Quaternion algebra has important applications in a variety of field in physics, biomedical, geometry, image processing and so on.
In the present section, to distinguish  quaternions from real numbers, we represent real number, real functions using normal letters and
quaternions and quaternionic functions using boldface letters, respectively.
Now, we begin by reviewing some basic definitions and properties of quaternion algebra.
The new numbers
$$
\bm{q} := q_0 + \mathbf{i} q_1 + \mathbf{j} q_2 + \mathbf{k} q_3 \quad (q_0,q_1,q_2,q_3\in \R),
$$
are called \emph{quaternions} (or more informally, Hamilton numbers) by W.R. Hamilton.
The set of all real quaternions is often denoted by $\H$, in honour of its discoverer.
The imaginary units $ \mathbf{i}$,  $\mathbf{j}$,  $\mathbf{k}$
obey the following laws of multiplication
\begin{eqnarray*}
\textbf{i}^2 = \textbf{j}^2 = \textbf{k}^2 = \textbf{i} \textbf{j} \textbf{k} = -1,
\end{eqnarray*}
and the usual component--wise defined addition. In particular,
the elements $\mathbf{i}$, $\mathbf{j}$, $\mathbf{k}$ are pairwise anticommute.

The neutral element of addition, known as {\it additive identity quaternion},
is defined by $\bm{0}:= 0+\mathbf{i}0+\mathbf{j}0+\mathbf{k}0$.
A quaternion can be also represented as $\bm{q}:=\mathbf{Sc}(\bm{q})+\underline{q}$,
where $\mathbf{Sc}(\bm{q})=q_0$  denotes the {\it scalar part} and
$\underline{q}:=\mathbf{i}q_1+\mathbf{j}q_2+\mathbf{k}q_3$ is the {\it vector part} of $\bm{q}$.
Like in the complex case, the {\it conjugate} of $\bm{q}$ is defined by $\overline{\bm{q}}:=\mathbf{Sc}(\bm{q})-\underline{q}$.
The {\it modulus} of $\bm{q}$ is defined as
\begin{eqnarray} \label{Modulus}
|\bm{q}| := \sqrt{\bm{q}\overline{\bm{q}}} = \sqrt{\overline{\bm{q}}\bm{q}} =\left(q_0^2+q_1^2+q_2^2+q_3^2\right)^{\frac{1}{2}}
\end{eqnarray}
and it coincides with its corresponding Euclidean norm as a vector in $\R^4$.

We consider a kind of hypercomplex signals $\H$-valued %$\bm{f}: \H^2 \to \H$ is a $\H$-valued signals with $\H$-valued variables $\bm{x}$ and $\bm{y}$.
%Here, $\bm{x}:=x_0+\mathbf{i}x_1+\mathbf{j}x_2+\mathbf{k}x_3$ and $\bm{y}:=y_0+\mathbf{i}y_1+\mathbf{j}y_2+\mathbf{k}y_3$.
%Let $X:=(x_0,x_1,x_2,x_3)$ and $Y:=(y_0,y_1,y_2,y_3)$ be two vectors.
%The $\bm{f}(\bm{x},\bm{y})$ can be expressed by
%$$\bm{f}(\bm{x},\bm{y})= f_0(X,Y) + \mathbf{i} f_1(X,Y) + \mathbf{j} f_2(X,Y) + \mathbf{k}f_3(X,Y),$$
%where $ f_i(X,Y): \R^8 \to \R$ $(i = 0,1,2,3)$ are real-valued functions.
%For simplicity, we consider a kind of hypercomplex signals $\H$-valued signals with $X:=(x,0,0,0)$ and $Y:=(y,0,0,0)$,
%i.e.
signals of the form $\bm{f}: \R^2 \to \H$ such that
\begin{eqnarray} \label{QuaternionForm}
 \bm{f}(x,y) &=& f_0(x,y) + \mathbf{i} f_1(x,y) + \mathbf{j} f_2(x,y) + \mathbf{k}f_3(x,y), \\
 f_i(x,y)&:& \R^2 \to \R \qquad (i = 0,1,2,3) \nonumber.
\end{eqnarray}
Properties (like integrability, continuity or differentiability) that are ascribed to $\bm{f}$ have to be fulfilled by all components $f_i$.
Since $\mathbf{k}=\mathbf{i}\mathbf{j}$, we can also rewrite Eq. (\ref{QuaternionForm}) as following
\begin{eqnarray} \label{QuaternionFormij}
\bm{f}(x,y) = f_0(x,y) + \mathbf{i} f_1(x,y) +f_2(x,y) \mathbf{j}  + \mathbf{i}f_3(x,y)\mathbf{j},
\end{eqnarray}
which keeps all  the imaginary unit $\mathbf{i}$ to the left and $\mathbf{j}$ to the right of each term \cite{H2007}.
Let $\mathcal{L}^p(\R^2;\H)$ $(p=1,2)$ denote the linear spaces of
all $\H$-valued functions in $\R^2$ under left multiplication by
quaternions such that each component is in the usual $\mathcal{L}^p(\R^2)$:
\begin{eqnarray}
\mathcal{L}^p(\R^2;\H) :=
\left\{\bm{f}\Big|\bm{f}:\R^2 \to \H , \; \int_{\R^2}  |\bm{f}(x,y)|^p dxdy < \infty \right\}.
\end{eqnarray}
We further introduce the left quaternionic inner product for two functions $\bm{f},\bm{g} \in \mathcal{L}^2(\R^2;\H)$ as follows
\begin{eqnarray} \label{InnerProduct}
<\bm{f},\bm{g}> \,:= \int_{\R^2} \bm{f}(x,y)\overline{\bm{g}(x,y)}dxdy.
\end{eqnarray}
Here, we just consider the left inner product throughout the paper.
Because the right inner product will lead the similar results in the following.
The reader should note that the norm induced by this inner product as follows
\begin{eqnarray} \label{Norm}
\|\bm{f}\|_{\mathcal{L}^2(\R^2; \H)} \,:=
\Big(< \bm{f},\bm{f}>\Big)^{1/2}=
\left(\int_{\R^2} |\bm{f}(x,y)|^2 dx dy \right)^{1/2}
\end{eqnarray}
coincides with the usual $\mathcal{L}^2$-norm of $\bm{f}$, considered as a vector-valued function.
In particular,
we note the symbol $\|\bm{f}\|_{\mathcal{L}^2}=\|\bm{f}\|_{\mathcal{L}^2(\R^2; \H)}$
for simplicity. The square norm $\|\bm{f}\|^2_{\mathcal{L}^2} :=E$
 is also often called the {\it total energy} of the $\H$-valued signal $\bm{f}$ and is normalized so that $E=1$ throughout.

The space $\mathcal{L}^2(\R^2;\H)$ furnished with the inner product Eq. (\ref{InnerProduct})
is an $\H$-valued Hilbert space and the norm in Eq. (\ref{Norm}) turns
$\mathcal{L}^2(\R^2, \H)$ into a Banach space \cite{BDS1982}.
We will make use of a special notation to express the scalar part of Eq. (\ref{InnerProduct}):
\begin{eqnarray} \label{Real_Inner_Product}
<\bm{f},\bm{g}>_0 \,:= \mathbf{Sc}\Big(<\bm{f},\bm{g}>\Big).
\end{eqnarray}
The real-valued inner product Eq. (\ref{Real_Inner_Product}) appeared
for example in \cite{D1975} in the context of complex vector spaces,
in \cite{GS1989} for spaces of $\H$-valued functions and it was also considered
in \cite{BDS1982} for spaces of Clifford-valued functions.

\begin{Def}
Two elements $\bm{f},\bm{g} \in \mathcal{L}^2(\R^2;\H)$ are orthogonal
in the $\mathcal{L}^2$-sense if $<\bm{f},\bm{g}>=\bm{0}$.
\end{Def}

To proceed with, we now define an angle between two $\H$-valued functions.
\begin{Def}
The angle between two non-zero functions $\bm{f}, \bm{g} \in {\mathcal {L}^2(\R^2;\H)}$ is defined by
\begin{eqnarray}
\arg(\bm{f},\bm{g}) :=
\arccos \left(\frac{< \bm{f},\bm{g}>_0}{\|\bm{f}\|_{\mathcal{L}^2} \|\bm{g}\|_{\mathcal{L}^2} }\right).
\end{eqnarray}
\end{Def}
The superimposed argument is well-defined since, obviously, it holds
$$|< \bm{f},\bm{g}>_0 |\leq |<\bm{f},\bm{g}>| \leq
\|\bm{f}\|_{\mathcal{L}^2} \|\bm{g}\|_{\mathcal{L}^2}.$$

The extremal values of  this angle will be discussed in Section \ref{S4} in detail.
%%%%%%%%%%%%%%Section2Section2Section2%%%%%%%%%%%%%%%%%%%%%%%%%%%%%%%%%%%%%%%%%%%%%%%%%%%%%%%%%%%%%%

\subsection{The Quaternionic Fourier Transform (QFT)}
The QFT we have used throughout the paper is as follows \cite{BHHA2008,H2007,HM2008}
\begin{Def}
Let $\bm{f}\in \mathcal {L}^1(\R^2;\H)$. The two-sided QFT of $\bm{f}$ is defined by
\begin{eqnarray}
\mathcal{F}(\bm{f})(u,v):=\frac{1}{2\pi} \int_{\R^2} e^{-\mathbf{i}ux}\bm{f}(x,y)e^{-\mathbf{j}vy}
dxdy,
\end{eqnarray}
where $(x,y)$, $(u,v)$ are points in $\R^2$.
\end{Def}
Here, $(x,y)$ will denote the {\it space} and $(u,v)$ the {\it angular frequency} variables.

Under suitable conditions, the original quaternionic function $\bm{f}$ can be reconstructed from $\mathcal{F}(\bm{f})$ by the inverse transform.
\begin{Def}
The inverse (two-sided) QFT of $\bm{f} \in \mathcal{L}^1 \bigcap \mathcal{L}^2(\R^2;\H)$, if applicable, is defined by
\begin{eqnarray}
\mathcal{F}^{-1}(\bm{f})(x,y) := \frac{1}{2\pi} \int_{\R^2} e^{\mathbf{i}ux}
\bm{f}(u,v)e^{\mathbf{j}vy} dudv.
\end{eqnarray}
\end{Def}

Since $\bm{f}= f_0 + \mathbf{i} f_1 +  f_2 \mathbf{j}+ \mathbf{i}f_3\mathbf{j}$ with
 $f_i: \R^2 \to \R,~(i = 0,1,2,3) $, the $\mathcal {F}(\bm{f})$ have a symmetric representation as follows
\begin{eqnarray}
\mathcal {F}(\bm{f})=
\mathcal {F}(f_0) +
\mathbf{i} \mathcal {F}(f_1)
 + \mathcal {F}(f_2)\mathbf{j}  +
  \mathbf{i}\mathcal {F}(f_3)\mathbf{j},
\end{eqnarray}
where $\mathcal {F}(f_i),~(i = 0,1,2,3) $ are $\H$-valued functions.
However, we can't use the modulus in Eq. (\ref{Modulus}) because the $\mathcal {F}(f_i),~(i = 0,1,2,3) $ are not real-valued functions.
For this reason, a $Q$ modulus of $\mathcal {F}(\bm{f})$ is introduced as follows \cite{CKL2015}
\begin{eqnarray}\label{Qmodulu}
|\mathcal {F}(\bm{f})|_Q^2
:=|\mathcal {F}(f_0)|^2+|\mathcal {F}(f_1)|^2+|\mathcal {F}(f_2)|^2+|\mathcal {F}(f_3)|^2.
\end{eqnarray}

Let $\bm{f}(x,y)$ and $\mathcal {F}(\bm{f})(u,v)\in \mathcal {L}^1(\R^2;\H)$.
The Parseval's theorem for the QFT is given as follows \cite{CKL2015}
\begin{eqnarray}\label{Parseval_Identity}
\|\bm{f}\|_{\mathcal{L}^2}^2 =\|\mathcal {F}(\bm{f}) \|_Q ^2,
\end{eqnarray}
where $\|\mathcal {F}(\bm{f}) \|_Q ^2:=\int_{\R^2}  |\mathcal {F}(\bm{f}) (u,v)|_Q ^2 dudv<\infty.$

From the Parseval's identity Eq. (\ref{Parseval_Identity}), we can find that
the total energy of a $\H$-valued signal in the time-domain is the same as that in the frequency-domain under the $Q$ modulus.

%%%%%%%%%%%%%%%%%%%%%%%%%%%section3%%%%%%%%%%%%%%%%%%%%%%%%%%%%%%%%%%%%%%%%%%%%%%
\section{The Quaternionic Prolate Spheroidal Wave Functions}
\label{S3}
The present section introduces the quaternionic prolate spheroidal wave functions (QPSWFs) and discusses some of their general properties.

First we introduce the notations we need in this part.
Let $\mathbf{T}:=[-T,  T] \times [-T,  T] \subset \R^2$ be the {\it time-domain},
and $\mathbf{W}:=[-W, W]\times [-W, W] \subset \R^2$ the {\it frequency-domain}
so that $\textbf{W}$ is a scaled version of $\textbf{T}$. We write $\textbf{W}=c\textbf{T}$,
where ${\bf{x}} \in c \textbf{T}$ if and only if ${\bf{x}}/c \in \textbf{T}$ with $c$ a positive constant.
For simplicity of presentation, we represent the integral notations $\int_{-T}^{T} \int_{-T}^{T}$ and $\int_{-W}^{W} \int_{-W}^{W}$
 in the abbreviated notation $\int_{\mathbf{T}}$ and $\int_{\textbf{W}}$, respectively.
Let $\mathcal {B}_\mathbf{W}
:=\{\bm{f}\in \mathcal {L}^2(\R^2;\H)~|~\mathcal {F}(\bm{f})(u,v)=0,  (u,v) \in
\R^2\setminus \mathbf{ W}\}$
be the Paley-Wiener space of $\H$-valued functions that are bandlimited to $\textbf{W}$.
The space $\mathcal {B}_\mathbf{W}$ will be discussed in detail in Section \ref{S4}.

\subsection{Definition of QPSWFs}
We are now ready to introduce the quaternionic prolate spheroidal wave functions in the finite quaternionic Fourier transform setting.
\begin{Def} [Finite-QFT form] \label{def1}
Given a real $c>0$, the quaternionic prolate spheroidal wave functions (QPSWFs)
$\bm{\psi}_{n}: \R^2 \rightarrow \H$ $(n=0, 1, \ldots)$, are the solutions of the integral equation
\begin{eqnarray} \label{Q1}
\mu_{n}
\mathbf{i}^n\bm{\psi}_{n}(x,y)\mathbf{j}^n
:=\int_{ \mathbf{ T}}
e^{\mathbf{i}csx}
\bm{\psi}_{n}(s,t)
e^{\mathbf{j}cty}
dsdt,
\end{eqnarray}
where $\mu_{n}$ are the complex parameters corresponding to the eigenfunction $\bm{\psi}_{n}(x,y)$.
Moreover, the functions $\{ \bm{\psi}_{n}(x,y) \}_{n=0}^{\infty}$ are complete in the class of $\textbf{W}$-bandlimited functions $\mathcal{B}_{\textbf{W}}$.
\end{Def}

In the notations used above we have concealed the fact that both the $\bm{\psi}_{n}(x,y)$'s and
the $\mu_{n}$'s depend on the parameter $c$. When it is necessary to make this dependence explicit,
we write $\mu_{n} = \mu_{n}(c)$ and $\bm{\psi}_{n}(x,y) = \bm{\psi}_{n}(x,y;c)$, $n=0,1, \cdots$, where $c>0$.
Naturally, considerable simplification occurs when $\bm{\psi}_{n}(x,y)$ is even or odd with $n$.
Due to the symmetry of the domain $\mathbf{T}$, one can easily shows that
$\bm{\psi}_{n}(-x,y)$, $\bm{\psi}_{n}(x,-y)$ and $\bm{\psi}_{n}(-x,-y)$ are also solutions
to  the integral Eq. (\ref{Q1}).  Consequently, they are solutions of Eq. (\ref{Q1}) as well
\begin{eqnarray*}
\bm{\psi}_{\textbf{ee}}(x,y)&:=&\bm{\psi}_{n}(x,y)+\bm{\psi}_{n}(-x,-y),\\
\bm{\psi}_{\textbf{e}^1}(x,y)&:=&\bm{\psi}_{n}(x,y)+\bm{\psi}_{n}(-x,y),\\
\bm{\psi}_{\textbf{e}^2}(x,y)&:=&\bm{\psi}_{n}(x,y)+\bm{\psi}_{n}(x,-y), \\
\bm{\psi}_{\textbf{o}^1}(x,y)&:=&\bm{\psi}_{n}(x,y)-\bm{\psi}_{n}(-x,y),\\
\bm{\psi}_{\textbf{o}^2}(x,y)&:=&\bm{\psi}_{n}(x,y)-\bm{\psi}_{n}(x,-y),\\
\bm{\psi}_{\textbf{oo}}(x,y)&:=&\bm{\psi}_{n}(x,y)-\bm{\psi}_{n}(-x,-y).
\end{eqnarray*}
The $\{\bm{\psi}_{n}(x,y)\}_{n=0}^{\infty}$ possess a number of special properties that make them most useful for the study of bandlimited functions.
They are also the eigenfunctions of the integral Eq. (\ref{Q2}) (see Theorem \ref{Theorem_2} below).

Before introduce the Theorem \ref{Theorem_2}, a lemma that connects sinc-functions and an integral is listed here without proof,
which can be easily checked.
\begin{Lem} \label{expart}
Let $(x,y)$, $(u,v)$ be points in $\R^2$. There holds
\begin{eqnarray}\label{Eq. expart}
\frac{1}{(2\pi)^2}\int_{\mathbf{W}}e^{\mathbf{i}u(x-s)}e^{\mathbf{j}v(y-t)}dudv
 =\frac{\sin W (x-s)}{\pi (x-s)}\frac{\sin W (y-t)}{\pi (y-t)}.
\end{eqnarray}
\end{Lem}

\begin{Th} [Low-pass filtering form] \label{Theorem_2}
The solutions of Eq. (\ref{Q1}) are also the solutions of the integral equation
\begin{eqnarray}\label{Q2}
\lambda_{n}\bm{\psi}_{n}(x,y):=\int_{ \mathbf{ T}}\bm{\psi}_{n}(s,t)
\frac{\sin W (x-s)}{\pi (x-s)}\frac{\sin W (y-t)}{\pi (y-t)}dsdt,
\end{eqnarray}
where $\lambda_{n}$ are eigenvalues corresponding to the eigenfunctions
$\bm{\psi}_{n}(x,y)$.
\end{Th}
\begin{proof}
Using Eq. (\ref{Eq. expart}) and having in mind that $\frac{\sin W (x_i-y_i)}{\pi (x_i-y_i)}$ $(i=1,2)$ is real value,
straightforward computations on the right side show that
\begin{eqnarray*}
&&\int_{ \mathbf{ T}}\frac{\sin W (x-s)}{\pi (x-s)}
\bm{\psi}_{n}(s,t)\frac{\sin W (y-t)}{\pi (y-t)}dsdt\\
=&&\frac{1}{(2\pi)^2}
\int_{ \mathbf{ T}}\int_{\mathbf{W}}
e^{\mathbf{i}u(x-s)}\bm{\psi}_{n}(s,t)e^{\mathbf{j}v(y-t)}dudvdsdt\\
=&&\frac{1}{(2\pi)^2}
\int_{\mathbf{W}}e^{\mathbf{i}ux}
\left(\int_{ \mathbf{ T}}e^{-\mathbf{i}us}
\bm{\psi}_{n}(s,t)
e^{-\mathbf{j}vt}
dsdt\right)e^{\mathbf{j}vy}dudv.
\end{eqnarray*}
Then we use the definition of QPSWFs to the integral above,
\begin{eqnarray*}
&&\frac{\mu_{n}}{(2\pi)^2}
\int_{\mathbf{W}}e^{\mathbf{i}ux}
\left((-\mathbf{i})^n\psi_{n}(\frac{u}{c},\frac{v}{c})(-\mathbf{j})^n\right)e^{\mathbf{j}vy}dudv\\
=&&\frac{c^2\mu_{n}}{(2\pi)^2}(-\mathbf{i})^n
\left(\int_{\mathbf{T}}e^{\mathbf{i}cu_1x}
\bm{\psi}_{n}(u_1,u_2)
e^{\mathbf{j}cu_2y}
du_1du_2\right)(-\mathbf{j})^n\\
=&&\frac{c^2\mu^2_n}{(2\pi)^2}(-\mathbf{i})^n\mathbf{i}^n
\bm{\psi}_{n}(x,y)\mathbf{j}^n(-\mathbf{j})^n
=\lambda_{n}\bm{\psi}_{n}(x,y).
\end{eqnarray*}
Derived from the above process, we find that  the relationships between $\mu_n$ and  $\lambda_n$ are
$\lambda_{n} := \frac{c^2\mu^2_{n}}{(2\pi)^2}$, $n=0,1, \dots$.
\end{proof}
Theorem \ref{Theorem_2} gives a straight forward derivation of QPSWFs from finite-QFT form to a low-pass filtering system.
The low-pass filtering form of QPSWFs is important to study the properties of them.
Since the noncommunicative of $\H$-valued signals with the QFT kernel,
the low-pass filtering form, which connects the $\H$-valued signals with two real-valued kernel,
provides an easy way to study the QPSWFs.
Most of the properties of QPSWFs are deduced from the low-pass filtering form.

\subsection{Properties of QPSWFs}
To state the properties of QPSWFs,  we shall need some new notations and basic facts about integral equations at first.
Let the real-valued kernel function
\begin{eqnarray} \label{k2}
k(x,y,s,t) := \frac{\sin W (x-s)}{\pi (x-s)}\frac{\sin W (y-t)}{\pi (y-t)}.
\end{eqnarray}
\begin{Def} [Quaternion Hermitian operator]
For any $\bm{f}\in \mathcal{L}^2(\R^2,\H)$, the quaternion Hermitian operator is defined by
\begin{eqnarray} \label{Quaternion_Hermitian_Operator}
(K\bm{f})(x,y) := \int_{ \mathbf{ T}}k(x,y,s,t) \bm{f}(s,t)dsdt.
\end{eqnarray}
\end{Def}
$K$ is a  Hermitian operator, because for any $\bm{f}, \bm{g}\in \mathcal{L}^2(\R^2,\H)$
$$
\int_{\R^2}\overline{\bm{f}(x,y)}(K\bm{g})(x,y)dxdy=\int_{\R^2}\overline{(K\bm{f})(x,y)}\bm{g}(x,y)dxdy.
$$
Here, since this operator acting on the $\H$-valued functions, we call it the quaternion Hermitian operator.
In particular, when the function is complex-valued , the role of this Hermitian operator is the original Hermitian operator.

For any $\bm{\psi} \in \mathcal{L}^2(\R^2,\H)$ the Hermitian operator Eq.
(\ref{Quaternion_Hermitian_Operator}) is indeed a linear combination of Hermitian operators of four real-valued signals
\begin{eqnarray}
\begin{array}{ll}
K\bm{\psi} = K\psi_0 + \mathbf{i}K\psi_1 + K\psi_2\mathbf{j} + \mathbf{i}K\psi_3\mathbf{j}=\lambda \bm{\psi}, \\
~ \psi_i: \R^2 \to \R \;\; (i = 0,1,2,3).
\end{array}
\end{eqnarray}
We can now use the properties of Hermitian kernels for real-valued signals to
study the properties of the Hermitian kernel function Eq. (\ref{k2}) and of its eigenvalues,
that is, values $\lambda$ of for which the homogeneous Fredholm integral equation
of the form $\lambda\varphi=K\varphi$ $(\varphi\in \R$) has a non-trivial solution.

From the general theory of integral equations,
we conclude the following properties for the kernel function \cite{K1992,M2012}:
\begin{enumerate}
\item[(i)] If $k(x,y,s,t) $ is a non-vanishing, continuous, Hermitian kernel,
then $k(x,y,s,t) $ has at least one eigenvalue;
\item[(ii)] If $k(x,y,s,t) $ is a complex-valued, non-vanishing, continuous,
Hermitian kernel, then the eigenvalues  associated to $k(x,y,s,t) $ are real;
\item[(iii)] Let $k(x,y,s,t) $ be a non-zero, continuous, Hermitian kernel.
If $\lambda_{m}$ and $\lambda_{n}$ are any two different eigenvalues,
then their corresponding eigenfunctions $\varphi_{m}$ and $\varphi_{n}$ are orthogonal;
\item[(iv)] \label{Lemorthonormal}
Let $k(x,y,s,t) $ be a non-zero, continuous, Hermitian kernel.
The set of all of its eigenfunctions $\{\varphi_{n}\}_{n=0}^{\infty}$ forms a mutually orthogonal system in $\mathbf{T}$:
\begin{eqnarray}\label{Hermitianorthonormal}
\int_{\mathbf{ T}}\varphi_{n}(x,y)\varphi_{m}(x,y) dx dy = \lambda_{n}\delta_{mn}.
\end{eqnarray}
Here, $\lambda\varphi=K\varphi$, for any $\varphi_{n}, \varphi_{m}\in \R$;
\item[(v)] Let $k(x,y,s,t) $ be a non-zero, continuous kernel defined in $\mathbf{T}$.
Let $\Lambda_k$ denote the set of eigenvalues of the kernel $k(x,y,s,t) $.
Then $\Lambda_k$ is at most countable, and it cannot have a finite limit point.
\end{enumerate}

\begin{Pro} \label{Main_Proposition}
Given a real $c>0$, we can find a countably infinite set of $\H$-valued signals $\{\bm{\psi}_{n}\}_{n=0}^{\infty}$
and a set of numbers $\{ \lambda_{n} \}_{n=0}^{\infty}$ with the following properties:
\begin{enumerate}
\item The eigenvalues $\lambda_n$'s are real and monotonically decreasing in $(0,1)$,
\begin{eqnarray}\label{eigenvaluerelation}
\lambda_{0} \geq\lambda_{1} \geq\lambda_{2} \geq \cdots,
\end{eqnarray}
and such that $\lim_{n\rightarrow \infty} \lambda_{n}=0$;
\item The $\{ \bm{\psi}_{n} \}_{n=0}^{\infty}$ are orthogonal in $\mathbf{T}$,
\begin{eqnarray}
\int_{ \mathbf{T}} \bm{\psi}_{n}(x,y) \overline{\bm{\psi}_{m}(x,y)} dxdy = 4\lambda_{n}\delta_{mn};
\end{eqnarray}
\item The $\{ \bm{\psi}_{n} \}_{n=0}^{\infty}$ are orthonormal in $\R^2$,
\begin{eqnarray}
\int_{\R^2} \bm{\psi}_{n}(x,y) \overline{\bm{\psi}_{m}(x,y)} dxdy= \delta_{mn}.
\end{eqnarray}
\end{enumerate}
\end{Pro}
\begin{proof}
Property 1. follows from the general theory of integral equations for
real-valued signals \cite{K1992,M2012} and is  stated without proof.

We now prove Statement 2. From Eq. (\ref{Hermitianorthonormal}) in Property (iv),
we know that the real-valued eigenfunctions
of the integral equation $\lambda\psi_i=K\psi_i$ $i=0,1,2,3$ form an orthonormal system in $\mathbf{T}$.
For the integral equation $K\bm{\psi}= \lambda\bm{\psi}$, ($\bm{\psi}\in \mathcal{L}^2(\R^2,\H)$),
we conclude that
\begin{eqnarray*}
&&\int_{ \mathbf{ T}}\bm{\psi}_{n}(x,y)\overline{\bm{\psi}_{m}(x,y)}dxdy\\
&= &\int_{ \mathbf{ T}}\left(\psi_{n,0} +\mathbf{i}\psi_{n,1} +
\psi_{n,2}\mathbf{j} +\mathbf{i}\psi_{n,3}\mathbf{j}\right)
\left(\psi_{m,0} -\mathbf{i}\psi_{m,1}-
\psi_{m,2}\mathbf{j} +\mathbf{i}\psi_{m,3}\mathbf{j}\right)dxdy\\
&=& 4\lambda_{n}\delta_{mn}=4\lambda_n\delta_{mn}.
\end{eqnarray*}
Here, we use the fact that for different eigenvalues  $\lambda_{n}$, $\lambda_{m}$, the eigenfunction are orthogonal, i.e.,
$
\int_{ \mathbf{ T}}\psi_{n,i}(x,y)\overline{\psi_{m,j}(x,y)}dxdy= \left\{
\begin{array}{ll}
\lambda_n, & m=n,i=j, i,j=0,1,2,3,\\
[1.5ex]
0,  &m\neq n,i,j=0,1,2,3.
\end{array}\right.
$

For  the  Statement 3, the orthogonality of the QPSWFs on the region $\R^2$ can be deduced as follows
\begin{eqnarray*}
&&\int_{\R^2}\bm{\psi}_{n}(x,y)\overline{\bm{\psi}_{m}(x,y)}dxdy\\
=&&\frac{1}{16\lambda_{n}\lambda_{m}}
\int_{\R^2}\int_{\mathbf{T}}\int_{\mathbf{T}}
\bm{\psi}_{n}(s,t)
\frac{\sin W (x-s)}{\pi (x-s)}\frac{\sin W (y-t)}{\pi (y-t)}\\
&&\frac{\sin W (x-z)}{\pi (x-z)}\frac{\sin W (y-w)}{\pi (y-w)}
\overline{\bm{\psi}_{m}(z,w)}
dsdtdzdwdxdy\\
=&&\frac{1}{16\lambda_{n}\lambda_{m}}
\int_{\mathbf{T}}\int_{\mathbf{T}}
\bm{\psi}_{n}(s,t)
\frac{\sin W (s-z)}{\pi (s-z)}\frac{\sin W (t-w)}{\pi (t-w)}
\overline{\bm{\psi}_{m}(z,w)}dsdtdzdw\\
=&&\frac{1}{4\lambda_{m}}
\int_{\mathbf{T}}\bm{\psi}_{n}(z,w)
\overline{v_{m}(z,w)}dzdw
=\frac{1}{4\lambda_{m}} 4\lambda_{m}\delta_{mn}
= \delta_{mn}.
\end{eqnarray*}
Here, we have substituted $\bm{\psi}_{n}(x,y)$ and $\bm{\psi}_{m}(x,y)$ by Eq. (\ref{Q2}) for the first and the third equality.
We use the equation $\frac{1}{2\pi}\int_R e^{\mathbf{i}(z-s)x}dx=\delta(z-s)$ for the second equality.
For the last equality, we utilize the orthogonality of PSWFs on the finite region.
Thus the orthogonality of the $\bm{\psi}_{n}$ over $\textbf{T}$ implies orthogonality over $\R^2$ and vice-versa.
\end{proof}

\begin{Rem}
Since $\|\bm{\psi}_{n}\|^2_{\mathcal{L}^2} = 1$ and
$\|\bm{\psi}_{n}\|^2_{\mathcal{L}^2({\bf{T}},\H)} = \lambda_{n}$,
a small value of $\lambda_{n}$ implies that $\bm{\psi}_{n}$ has most of
its energy outside the interval $\textbf{T}$ while a value of $\lambda_{n}$
near $1$ implies that $\bm{\psi}_{n}$ is concentrated largely in $\textbf{T}$.
\end{Rem}

\begin{Pro} [All-pass filtering form]
The QPSWFs satisfy the following integral equation
\begin{eqnarray}\label{allpass}
\bm{\psi}_{n}(x,y)=\int_{ \R^2}\bm{\psi}_{n}(s,t)
\frac{\sin W (x-s)}{\pi (x-s)}\frac{\sin W (y-t)}{\pi (y-t)}dsdt,
\end{eqnarray}
which is called the all-pass filtering form of QPSWFs.
\end{Pro}
\begin{proof}
This proposition can easy be obtain by the special case of the quaternion convolution theorem in \cite{MRR2013}.
Taking the QFT to the both sides of the Eq. (\ref{Q2}), we have
\begin{eqnarray*}
\lambda_n\F(\bm{\psi}_{n})(u,v)=\F(\bm{\psi}_{n}\chi_{\mathbf{ T}})(u,v)\chi_{\mathbf{W}}(u,v),
\end{eqnarray*}
where
$\chi_{\mathbf{ T}}(x,y):= \left\{
\begin{array}{ll}
1, & (x,y)\in \mathbf{ T},\\
[1.5ex]
0,  & {\rm otherwise}.
\end{array}\right.$
and  $\chi_{\mathbf{ W}}(u,v):= \left\{
\begin{array}{ll}
1, & (u,v)\in \mathbf{W},\\
[1.5ex]
0,  & {\rm otherwise}.
\end{array}\right.$
From this equation, we know that
$$
\F(\bm{\psi}_{n})(u,v)=\F(\bm{\psi}_{n})(u,v)\chi_{\mathbf{W}}(u,v).
$$
Then taking the inverse QFT on both sides, we have
\begin{eqnarray}
\bm{\psi}_{n}(x,y)=\int_{ \R^2}\bm{\psi}_{n}(s,t)
\frac{\sin W (x-s)}{\pi (x-s)}\frac{\sin W (y-t)}{\pi (y-t)}dsdt.
\end{eqnarray}
\end{proof}
This all-pass  filtering form of QPSWFs extent the space domain from $\mathbf{ T}$ to $\R^2$.

%%%%%%%%%%%%%%%%%%%%%%%%%%%section4%%%%%%%%%%%%%%%%%%%%%%%%%%%%%%%%%%%%%%%%%%%%%%
\section{The Time-limited and Bandlimited spaces}
\label{S4}
In the present section, we consider two kinds of $\H$-valued functions and their corresponding spaces.
\begin{Def}
We say that a $\H$-valued function $\bm{f}(x,y)$ with finite energy is time-limited
if it vanishes for all $(x,y) \in \R^2 \setminus \mathbf{T} =: \mathbf{ T^c}$.
\end{Def}
\begin{Def}
We say that a $\H$-valued function $\bm{f}(x,y)$ with finite energy is bandlimited
if $\mathcal{F}(\bm{f})(u,v) \equiv  \bm{0}$ for all $(u,v)
\in \R^2 \setminus \mathbf{W} =: \mathbf{W^c}$.
\end{Def}

For any function $\bm{f} \in \mathcal {L}^2(\R^2,\H)$,
we can define two operators such that $\bm{f}$ becomes either a time-limited and bandlimited function.
We define the time-limited operator  $D_\mathbf{T}$ as follows
\begin{eqnarray}
D_\mathbf{T}\bm{f}(x,y):=\bm{f}(x,y)\chi_{\mathbf{ T}}(x,y),
\end{eqnarray}
where
$
\chi_{\mathbf{ T}}(x,y):= \left\{
\begin{array}{ll}
1, & (x,y)\in \mathbf{ T},\\
[1.5ex]
0,  & {\rm otherwise}.
\end{array}\right.
$
The operator $D_\mathbf{T}$ is linear
$
D_\mathbf{T}(\alpha \bm{f}+\beta \bm{g})(x,y)=\alpha D_\mathbf{T}\bm{f}(x,y)+\beta D_\mathbf{T}\bm{g}(x,y),
$
for any $\alpha, \beta \in \R$. Hence there holds
$
< D_\mathbf{T}\bm{f}, D_\mathbf{T^c}\bm{f}  >
= \bm{0},~ \mathrm{for~ all}~ \bm{f} \in \mathcal {L}^2(\R^2,\H).
$

The bandlimited operator  $B_\mathbf{W}$ is
$
B_{\mathbf{W}}\bm{f}(x,y) := \frac{1}{2\pi }\int_{\mathbf{W}}e^{\mathbf{i}ux}
\mathcal{F}(\bm{f})(u,v)e^{\mathbf{j}vy}dudv.
$
Noted that $B_\mathbf{W}$ can be decomposed as follows
\begin{eqnarray*}
B_{\mathbf{W}}\bm{f}(x,y)
&=&\mathcal {F}^{-1}\Bigl(D_\mathbf{W}\mathcal {F}(\bm{f})(u,v)\Bigr)\\
&=&\mathcal {F}^{-1}\Bigl(D_\mathbf{W}\left( \mathcal {F}(f_0)+\mathbf{i}\mathcal {F}(f_1)+
\mathcal {F}(f_2)\mathbf{j}+\mathbf{i}\mathcal {F}(f_3)\mathbf{j}\right)\Bigr)\\
&=& B_{\mathbf{W}}f_0+\mathbf{i}B_{\mathbf{W}}f_1+B_{\mathbf{W}}f_2\mathbf{j}+\mathbf{i}B_{\mathbf{W}}f_3\mathbf{j}.
\end{eqnarray*}
Here,  $f_i: \R^2 \to \R,~(i = 0,1,2,3)$.
Since the QFT is a linear transform, it follows that the operator $B_\mathbf{W}$ is also linear.
Similarly, there holds
\begin{eqnarray}
< B_\mathbf{W}\bm{f}, B_\mathbf{\mathbf{W}^c}\bm{f}>
= \bm{0}, ~\mathrm{for~ all}~
\bm{f}\in \mathcal {L}^2(\R^2,\H).
\end{eqnarray}

\begin{Rem}
For real-valued functions $f_i(x,y), i=0,1,2,3$, the inverse  QFT
$\mathcal {F}^{-1}(D_\mathbf{W}\mathcal {F}(f_i)(u,v))=B_{\mathbf{W}}f_i$.
In most cases,  $\mathcal {F}^{-1}(D_\mathbf{W}\mathcal {F}(f_i)), i=0,1,2,3$
are $\H$-valued functions, but here we conclude that
it is real-valued function. Since
\begin{eqnarray*}
B_{\mathbf{W}}f_i(x,y)
&=&\frac{1}{2\pi }\int_{\mathbf{W}}e^{\mathbf{i}ux}
\mathcal {F}(f_i)(u,v)e^{\mathbf{j}vy}dudv\\
&=&\frac{1}{(2\pi)^2 }\int_{\mathbf{W}}\int_{\R^2}e^{\mathbf{i}ux}
e^{-\mathbf{i}us}f_i(s,t)e^{-\mathbf{j}vt}
e^{\mathbf{j}vy}dsdtdudv\\
&=&\int_{\R^2}\frac{\sin W(x-s)}{\pi(x-s)}f_i(s,t)\frac{\sin W(y-t)}{\pi(y-t)}dsdt,
\end{eqnarray*}
form the third equality, we can find that $B_{\mathbf{W}}f_i(x,y)$ is
just a filter transform for these real-valued functions $f_i(x,y), i=0,1,2,3$.
That means $B_{\mathbf{W}}f_i(x,y)$ are also real-valued functions.
\end{Rem}

Now, we discuss the spaces induced by these two kinds of  operators.
The space induced by $D_\mathbf{T}$ is
\begin{eqnarray}
\mathcal {D}_\mathbf{T} := \{\bm{f} \in \mathcal{L}^2(\R^2,\H)
~|~\bm{f}(x,y) \equiv \bm{0}, (x,y) \in \mathbf{T^c} \}.
\end{eqnarray}
It is easy to see  that $\mathcal{D}_\mathbf{T}$ is a linear subspace of $\mathcal {L}^2(\R^2,\H)$.
\begin{Lem}
$\mathcal {D}_\mathbf{T}$ is complete.
\end{Lem}
\begin{proof}
We must show that if a set of $\H$-valued functions $\{\bm{f}_n(x,y)\}_{n=1}^{\infty}$
is a Cauchy sequence in $\mathcal {D}_\mathbf{T}$, then it converges  in $\mathcal {D}_\mathbf{T}$.
Due to the completeness of $\mathcal {L}^2(\R^2,\H)$,
there is a function $\bm{f}^* \in \mathcal {L}^2(\R^2,\H)$ such that
$\|\bm{f}_n-\bm{f}^*\|_{\mathcal{L}^2} \rightarrow  \bm{0}$  as such $n \rightarrow \infty$.
We  now define  the new function
\begin{eqnarray}
D_{\mathbf{ T}}\bm{f}^*(x,y) = \bm{f}^*(x,y)\chi_{\mathbf{ T}}(x,y),
\end{eqnarray}
where $\chi_{\mathbf{ T}}$ is the characteristic function on $\mathbf{T}$.
It can be verified that $\|\bm{f}_n-D_{\mathbf{T}}\bm{f}^*\|_{\mathcal{L}^2} \rightarrow  \bm{0}$ as
$n$ approaches infinity and, moreover,  $D_{\mathbf{T}}\bm{f}^*\in \mathcal{D}_\mathbf{T}$.
Thus, $\mathcal {D}_\mathbf{T}$ is complete.
\end{proof}

Similarly, the space induced by $B_\mathbf{W}$
form a linear subspace $\mathcal {B}_\mathbf{W}$ of $\mathcal {L}^2(\R^2,\H)$, which is also complete.

Now we consider the extremal values of the angle between  time-limited and  bandlimited  functions.
 The question then arises as to  what are the extremal values of $\arg(\bm{f},\bm{g})$
 between $\bm{g} \in \mathcal{D}_{\mathbf{T}}$ and $\bm{f} \in \mathcal{B}_\mathbf{W}$ under the QFT?
The following lemma and theorem find this extremal value. Since the proof is analogous to the one-dimensional case in \cite{LP1961}.
We restate the similar results for $\H$-valued functions without proof.
\begin{Lem}\label{Lemtheta}
If $\bm{f}\in \mathcal {B}_\mathbf{W}$ is fixed,  then $\arg(\bm{f},\bm{g})$ between $\bm{f}$ and any $\bm{g} \in \mathcal{D}_\mathbf{T}$
satisfies $\inf _{\bm{g}\in \mathcal{D}_{\mathbf{ T}}}\arg( \bm{f},\bm{g})>0$.  This infimum equals
$\arccos \frac{\|D_{\mathbf{T}}\bm{f}\|_{\mathcal{L}^2}}{\|\bm{f}\|_{\mathcal{L}^2}}$
 and is assumed by $\bm{g}=kD_{\mathbf{T}}\bm{f}$ for any  positive constant $k$.
\end{Lem}

We proceed now to find the least $\arg(\bm{f},\bm{g})$ of arbitrary $\bm{f} \in \mathcal {B}_\mathbf{W}$ and $\bm{g}\in \mathcal {D}_\mathbf{T}$.
We show that the spaces $\mathcal {B}_\mathbf{W}$ and $\mathcal {D}_\mathbf{T}$ form indeed a least angle.
\begin{Th}\label{Thangle}
Let $\bm{f} \in \mathcal {B}_\mathbf{W}$ and $\bm{g} \in \mathcal {D}_\mathbf{T}$.
There exists a least angle between $\mathcal {B}_\mathbf{W}$ and $\mathcal {D}_\mathbf{T}$ that satisfies
\begin{eqnarray}
\inf_ {\bm{f}\in \mathcal {B}_\mathbf{W}, \, {\bm{g}\in \mathcal {D}_\mathbf{T}}} \arg(\bm{f},\bm{g})=\arccos  \sqrt{\lambda_{0}}
\end{eqnarray}
if and only if $\bm{f}=\bm{\psi}_{0}$  and  $\bm{g}=D_{\mathbf{T}}\bm{\psi}_0$, where $\lambda_{0}$ is the largest eigenvalue of (\ref{Q2}),
and $\bm{\psi}_{0}$ is the corresponding eigenfunction.
\end{Th}
The proof is analogous to the one-dimensional case \cite{LP1961}, which is omitted here.

We have thus found that the two subspaces $\mathcal {B}_\mathbf{W}$
and $\mathcal {D}_\mathbf{T}$ of $\mathcal{L}^2(\R^2,\H)$,
which have no functions except $\bm{0}$ in common,
actually have a minimum angle between them, so that, in fact,
a time-limited function and a bandlimited function cannot even be very close together.

%%%%%%%%%%%%%%%%%%%%%%%%%%%section5%%%%%%%%%%%%%%%%%%%%%%%%%%%%%%%%%%%%%%%%%%%%%%
\section{The Energy Extremal Theorem}
\label{S5}
In the present section, we develop the relationship between the energy concentration
of a $\H$-valued signal in the time and frequency spaces. In particular,
if a finite energy $\H$-valued signal is given, the possible proportions of
its energy in a finite time-domain and a finite frequency-domain are found,
as well as the signals which are the most optimal energy concentration signals.

Throughout this section let $\bm{g}(x,y)$ be a $\H$-valued function,
$\mathbf{T}$ and $\mathbf{W}$ be two specified square regions
such that $\mathbf{W}=c\mathbf{T}$ with $c$ a positive constant.
We form the energy ratios in the time-domain and frequency-domain, respectively, as follows
\begin{eqnarray}
\xi ^2:=\frac{\|D_{\mathbf{T}}\bm{g}\|_{\mathcal{L}^2}^2 }
{\|\bm{g}\|_{\mathcal{L}^2}^2}
\quad {\rm and} \quad
\eta_Q ^2:= \frac{\|D_{\mathbf{W}}\mathcal {F}(\bm{g})\|_Q^2}{\|\mathcal {F}(\bm{g})\|_Q^2}.
\end{eqnarray}
Assume for simplicity that $E := \|\bm{g}\|_{\mathcal{L}^2}^2 =1$ in the following.
As $\bm{g}(x,y)$ ranges over all functions in $\mathcal {L}^2(\R^2,\H)$,
the pair $(\xi,\eta_Q)$ takes various values of the unit square in the $(x,y)$-plane.
In the following, we will discuss the set of possible pairs $(\xi,\eta_Q)$.
In particular, we will determine the functions that reach the extremal values of $(\xi,\eta_Q)$.

\begin{Lem}\label{Lementire}
A non-zero $\H$-valued function $\bm{g}(x,y)\in \mathcal {L}^2(\R^2,\H)$
cannot be both bandlimited and  time-limited.
\end{Lem}
\begin{proof}
From the uncertainty principle of  \cite{DS1989,CKL2015},
 it follows that $|\mathbf{T}| |\mathbf{W}|\geq(1-\varepsilon_\mathbf{T}-\varepsilon_\mathbf{W})^2,$
where $\bm{g}(x,y)$  has an essential support on $\mathbf{T}$ and $\mathbf{W}$.
In our  case $\varepsilon_\mathbf{T}=0$ and $\varepsilon_\mathbf{W}=0$,
and therefore, $ |\mathbf{T}| |\mathbf{W}|\geq 1$  holds.
That means, if a non-zero $\bm{g}(x,y)$ is bandlimited,
then $\bm{g}(x,y)$ cannot be identically zero in the region $\mathbf{T}$.
\end{proof}

\begin{Rem}
If the $\H$-valued function is bandlimited, i.e. $\eta_Q =1$, then $\xi \neq 0$ by Lemma \ref{Lementire}.
Analogously, if the function is time-limited, i.e. $\xi  =1$, then $\eta_Q \neq 0$.
\end{Rem}

We  now discuss the possible values of $(\xi,\eta_Q)$  if the underlying signal is bandlimited.
\begin{Th}\label{Thbandlimit}
For any non-zero bandlimited function $\bm{g}(x,y) \in \mathcal {B}_\mathbf{W}$,
 there holds $\xi\leq\sqrt{\lambda_0}$.
\end{Th}
\begin{proof}
If a  $\H$-valued function is bandlimited, then $\eta_Q=1$.
Since $\{\bm{\psi}_n(x,y)\}_{n=0}^\infty$ is a complete basis in
$\mathcal {B}_\mathbf{W}$ we can expanse $\bm{g}(x,y)$ as follows
\begin{eqnarray}
\bm{g}(x,y)=\sum_{n=0}^\infty a_n \bm{\psi}_{n} (x,y),
\end{eqnarray}
where  $a_n \in \H$. Clearly,
$
\xi^2=\int_{\mathbf{T}}|\bm{g}(x,y)|^2dxdy=
\sum_{n=0}^\infty \lambda_n|a_n|^2=
1 \leq \lambda_0\sum_{n=0}^\infty |a_n|^2,
$
because $\lambda_0\geq\lambda_n$  for all $n$.
Hence, the extremal condition is $\bm{g}^*(x,y)=\bm{\psi}_0(x,y)$.
It follows that $\xi^2=|\bm{g}^*|^2=\lambda_0$.
For any other linear combination of $\bm{\psi}_n$,  $\xi$ is less than $\sqrt{\lambda_0}$.
\end{proof}
As we know, $\xi$ take values in the interval $[0,1]$.
Now we discuss the range of values of $\eta_Q$ when $\xi$ is fixed.
We start by considering the case $\xi=0$.

\begin{Th}\label{timezero}
If the time energy ratio $\xi$  equals to $0$ on $\mathbf{T}$,
the frequency energy ratio $\eta_Q$ is larger or equals to $0$ and
less than $1$  on $\mathbf{W}$ for a $\H$-valued signal.
\end{Th}
\begin{proof}  See Appendix A.
\end{proof}

From the property of symmetry of the QFT we conclude that
all the properties for bandlimited functions proved so far have corresponding time-limited counterparts.
For example, if $\xi=0$,  then we can get $0\leq\eta_Q <1$.
If $\eta_Q=0$,  then it follows that $0\leq\xi <1$.
We can conclude a similar result to Theorem \ref{Thbandlimit} for the extremal value $\xi=1$.
For any non-zero time-limited function $\bm{g}(x,y)\in \mathcal {D}_\mathbf{T}$,
i.e.  for which $\xi=1$, we conclude that $\eta_Q\leq\sqrt{\lambda_0}$.
If $\eta_Q=\sqrt{\lambda_0}$, then $\bm{g}^*(x,y)=\frac{D_{\mathbf{T}}\bm{\psi}_0(x,y)}{\sqrt{\lambda_0}}$.

We now prove that for arbitrary $\H$-valued signals for which $0<\xi<\sqrt{\lambda_0}$, $\eta_Q$ is not limited.
\begin{Lem}
For $0< \xi < \sqrt{\lambda _{0}}$, $\eta_Q$ can take any value in the interval $[0,1]$ for non-zero $\H$-valued signals.
\end{Lem}
\begin{proof}
For $0< \xi < \sqrt{\lambda _{0}}$, the sequence $\{\lambda_{n} \}_{n=1}^{\infty}$ is
monotone decreasing in the interval $(0,1)$, and $\lambda_{n} \rightarrow 0$ when $n$ approaches infinity.
 Hence there exists an eigenvalue  such that $\lambda_{n} <\xi$. We consider the signal
\begin{eqnarray}
\bm{g}^*(x,y) = \frac{\sqrt{\xi^2-\lambda_{n} }\bm{\psi}_{0} (x,y) + \sqrt{\lambda_{0} -\xi^2 }\bm{\psi}_{n} (x,y)}{\sqrt{\lambda_{0}-\lambda_{n}}},
\end{eqnarray}
where $\bm{\psi}_{n}(x,y)$ is the eigenfunction corresponding to the eigenvalue $\lambda_{n}$.
We can find  that $\bm{g}^*(x,y) \in \mathcal {B}_\mathbf{W}$ since
$\bm{\psi}_{0} (x,y), \bm{\psi}_{n} (x,y) \in \mathcal {B}_\mathbf{W}$.
 Straightforward computations show that
\begin{eqnarray*}
\|\bm{g}^*\|_{\mathcal{L}^2}^2 &=& \frac{1}{\lambda_{0} -
\lambda_{n}} \Big[(\xi^2-\lambda_{n}) + (\lambda_{0} -\xi^2) \Big] = 1,\\
\|D_{\mathbf{T}}\bm{g}^*\|_{\mathcal{L}^2}^2
&=&
\frac{1}{\lambda_{0} - \lambda_{n}} \Big[(\xi^2-\lambda_{n}) \lambda_{0} + (\lambda_{0} - \xi^2)\lambda_{n} \Big] = \xi^2,
\end{eqnarray*}
which implies that $\bm{g}^*(x,y)\in \mathcal{A}$.  We conclude that
\begin{eqnarray*}
\eta_Q^2&=& \|\mathcal {F}(B_{\mathbf{W}}\bm{g}^*)\|_Q^2
=\sum_{i=0}^3\int_{\mathbf{W}}|\mathcal {F}(g^*_i)|^2 dudv\\
&=&\sum_{i=0}^3\int_{\R^2}|B_{\mathbf{W}}g^*_i|^2dxdy
=\frac{4}{4(\lambda_{0}-\lambda_{n})} \Big[(\xi^2-\lambda_{n})+(\lambda_{0} -\xi^2)\Big]=1.
\end{eqnarray*}
Here, the last equality used the following result that  for $i=0,1, 2, 3$,
\begin{eqnarray*}
\int_{\R^2}|B_{\mathbf{W}}g^*_i|^2dxdy
&=&\frac{1}{4(\lambda_{0}-\lambda_{n})}\int_{\R^2}
\Big(\sqrt{\xi^2-\lambda_{n} }\psi_{0,0}(x,y)+\sqrt{\lambda_{0} -\xi^2 }\psi_{n,0} (x,y)\Big)\\
&& \overline{\Big(\sqrt{\xi^2-\lambda_{n} }\psi_{0,0} (x,y) +
\sqrt{\lambda_{0} -\xi^2 }\psi_{n,0} (x,y)\Big)}dxdy\\
&=&\frac{1}{4(\lambda_{0}-\lambda_{n})} \Big(
\int_{\R^2}(\xi^2-\lambda_{n})\psi_{0,0} (x,y)\overline{\psi_{n,0} (x,y)}dxdy\\
&& -\int_{\R^2}\sqrt{(\xi^2-\lambda_{n})(\lambda_{0} -\xi^2) }\psi_{0,0} (x,y)\overline{\psi_{n,0} (x,y)}dxdy\\
&& -\int_{\R^2}\sqrt{(\xi^2-\lambda_{n})(\lambda_{0} -\xi^2) }\psi_{n,0} (x,y)\overline{\psi_{0,0} (x,y)}dxdy\\
&&+\int_{\R^2}(\lambda_{0} -\xi^2)\psi_{n,0} (x,y)\overline{\psi_{n,0} (x,y)}dxdy)\Big)\\
&=&\frac{(\xi^2-\lambda_{n})+(\lambda_{0} -\xi^2)}{4(\lambda_{0}-\lambda_{n})}.
\end{eqnarray*}
Thus, if $0< \xi < \sqrt{\lambda _{0}}$,  then there exists a signal such that $\eta_Q = 1$.
The verification of $0 \leq \eta_Q <1$ is similar to Theorem \ref{timezero}.
\end{proof}

We conclude this section by studying the range of possible values of $\eta$ for which $\sqrt{\lambda _{0}}\leq \xi<1$.
\begin{Th}\label{ThUP}
The maximum of $\eta_Q$ is  assumed by
\begin{eqnarray}
\arccos \xi +\arccos \eta_Q = \arccos \sqrt{\lambda _{0}},
\end{eqnarray}
as such $\sqrt{\lambda _{0}}\leq \xi < 1$,  where $\lambda _{0}$ is the largest eigenvalue of Eq. (\ref{Q2}).
\end{Th}
\begin{proof} See Appendix B.
\end{proof}
Here, the result is the same as the real-valued case. But the functions we considered are $\H$-valued.
That means, we generalized the classical results to covered a bigger space.

\begin{Rem}
From  Theorem \ref{ThUP} it follows that the set of all possible pairs $(\xi,\eta_Q)$ is the region
enclosed by Figure \ref{fig.relationship} bounded by the curve $\arccos \xi +\arccos \eta_Q \geq \arccos \sqrt{\lambda _{0}}$.
Although this curve depends on  the parameter $c$,  for simplicity of presentation we take $c=1$.
\begin{figure}[!h]
  \centering
    \includegraphics[width=10cm]{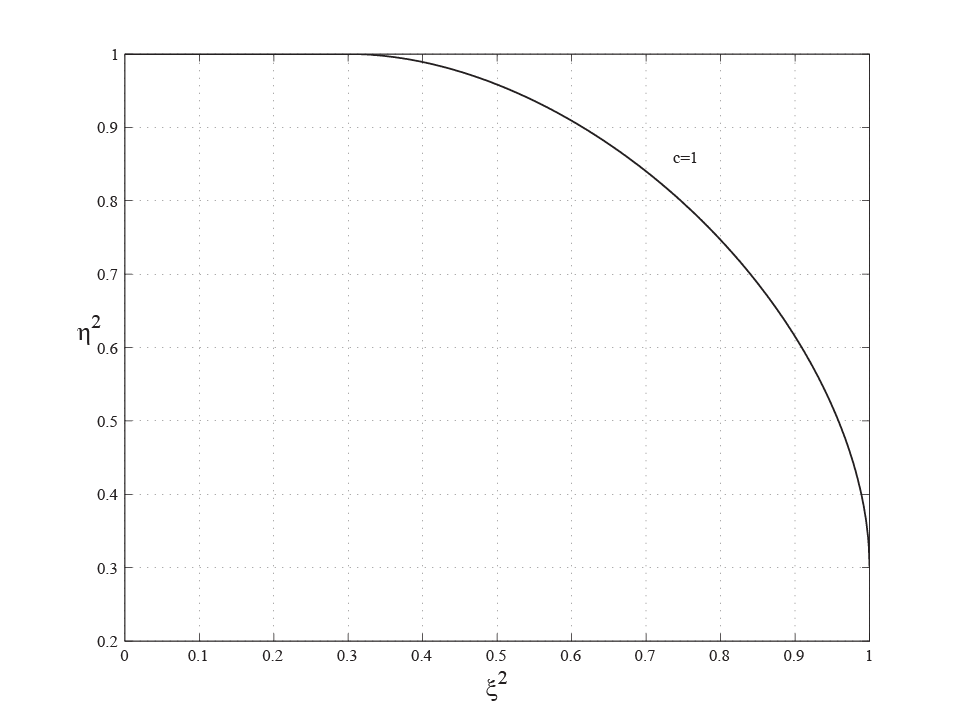}
  \caption{The relationship between $\xi^2$ and $\eta_Q^2$ for $c=1$.}
 \label{fig.relationship}
\end{figure}
\end{Rem}

\begin{Rem}
By means of the classical one-dimensional PSWFs we now construct a special case of a QPSWFS as follows
\begin{eqnarray}
\bm{\psi}_0(x,y)=\frac{1}{2}\varphi_0(x) \varphi_0(y)(1+\mathbf{i}+\mathbf{j}+\mathbf{k}),
\end{eqnarray}
where $\varphi_0$ is the first real-valued PSWFs of zero order in one-dimension.

\begin{table*}[ih]
\caption{The values of $(\xi, \eta_Q)$ and $\lambda$ for the first six QPSWFs.}
\bigskip\renewcommand\arraystretch{1.5}
\begin{center}
\begin{tabular}
{c|c|c|c|c|c|c}
\hline\cline{1-7}
${\rm Functions}$      &$\bm{\psi}_0$  &$\bm{\psi}_1$  &$\bm{\psi}_2$  &$\bm{\psi}_3$  &$\bm{\psi}_4$  &$\bm{\psi}_5$  \\
\hline
$\lambda$    &0.9999   &0.9951  &0.9206  &0.5211   &0.0754  &0.0018 \\
$\xi$             & 0.9194   & 0.5883  &0.2445  &0.2335   &0.2221  &0.2050 \\
$\eta_Q$          &1.0000   &0.9998  &1.0000  &0.9862   &1.0000  &0.9381 \\
\hline\cline{1-7}
\end{tabular}
\end{center}
\label{tab:1}
\end{table*}

We will compare six pairs of $(\xi,\eta_Q)$ of most energy concentrate function under the condition $c=1$.
For example, we compare the signals $\bm{\psi}_0(x,y)$ and $\bm{\psi}_5(x,y)$.
The signal $\bm{\psi}_0(x,y)$ is the most energy concentrate function in the frequency domain,
while $\bm{\psi}_5(x,y)$ is not concentrate in the centre of the frequency domain.
In Table \ref{tab:1} we list the first six pairs of $(\xi,\eta_Q)$ and $\lambda$ for the first QPSWFs.
The energy of $\bm{\psi}_0(x,y)$, $\bm{\psi}_2(x,y)$ and $\bm{\psi}_4(x,y)$ in the frequency domain is $1$,
while the others are less than $1$.
This is mainly because these  signals $\bm{\psi}_i$ $(i=0,2,4, \dots)$ have a single peak in the frequency domain,
and the corresponding energy is mainly concentrate in this peak.
For $\bm{\psi}_i$, $i=1,3,5, \dots$, there are four peaks.
In this case the energy is dispersed in a larger bandwidth.
As shown in Figure \ref{fig.QPSWFS},  the function $\bm{\psi}_5(x,y)$  has four peaks in the frequency
domain.
\begin{figure}[!h]
  \centering
    \includegraphics[width=6.0cm]{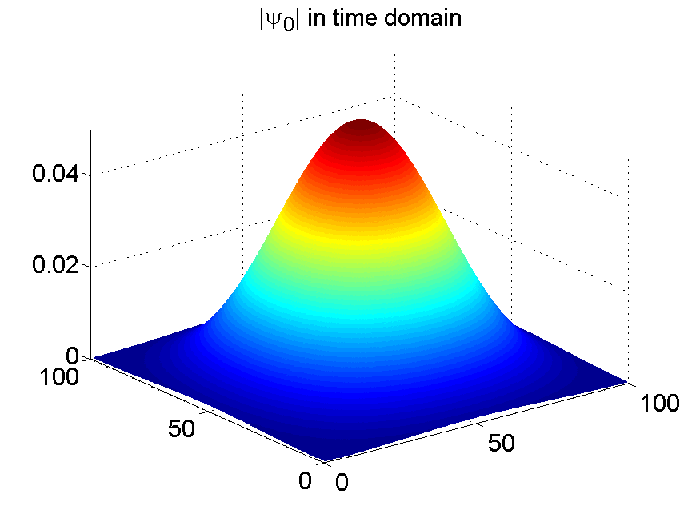}
    \includegraphics[width=6.0cm]{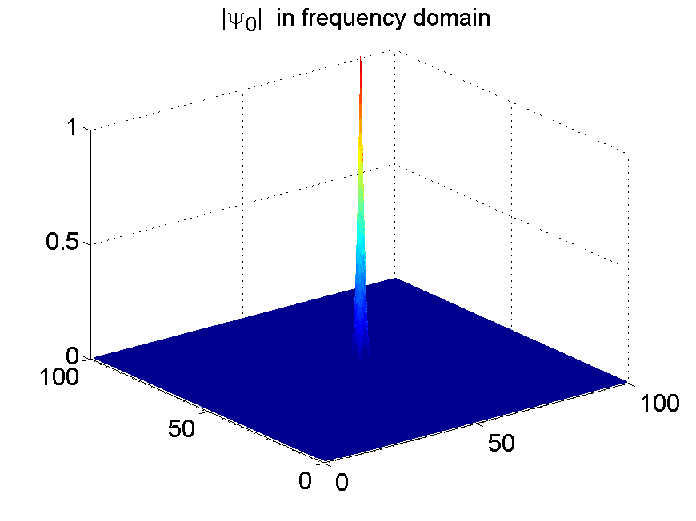}
    \par
    \includegraphics[width=6.0cm]{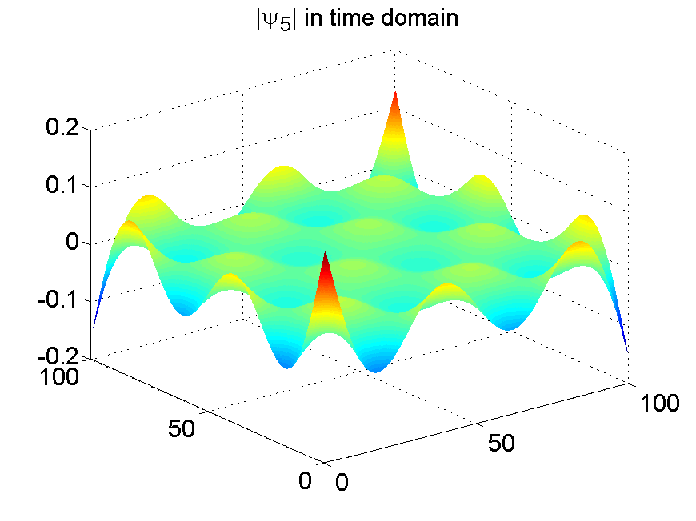}
    \includegraphics[width=6.0cm]{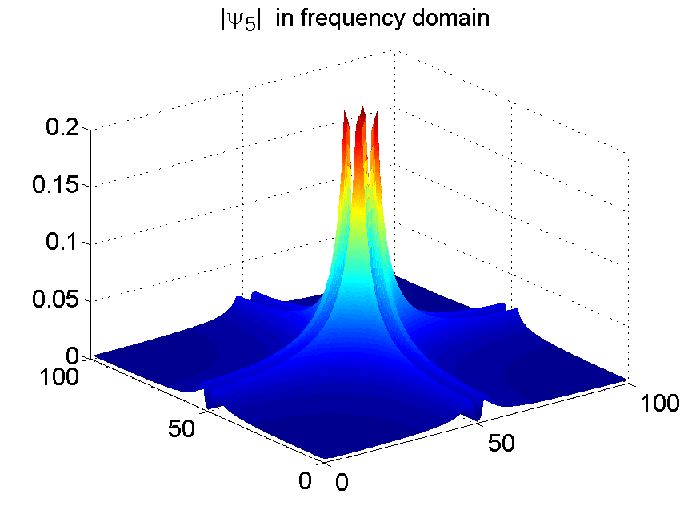}
  \caption{The modulus of QPSWFs $\psi_0(x,y)$ and $\psi_5(x,y)$ in  time and frequency domains. }
 \label{fig.QPSWFS}
\end{figure}

The results in Table \ref{tab:1} also support  the obtained energy concentration proportions of
a signal as such $\arccos \xi +\arccos \eta \geq \arccos \sqrt{\lambda _{0}}$.
In fact, when $\arccos\eta=\arccos1=0$, it follows that the corresponding smallest value is
$\arccos\xi=\arccos 0.9194 =0.4042 > \arccos \sqrt{\lambda_0} =0.0141$.
\end{Rem}

\section{QPSWFs in Bandlimited Extrapolation Problem}
\label{S6}
In the present paper, the QPSWFs have found to be the most energy concentration signals and for any
$\mathbf{W}$-bandlimited $\H$-valued signal $\bm{f}$ can be expanded into a series
\begin{eqnarray}
\bm{f}(x,y)=\sum_{j=0}^{\infty}a_j\bm{\psi}_n(x,y).
\end{eqnarray}
However, the $\mathbf{W}$-bandlimited $\bm{f}$ is finite segment in space domain empirically
(assume the domain known is $\mathbf{D}:=[-d,d]\times[-d,d]$), i.e.,
\begin{eqnarray}
\bm{g}(x,y)=\bm{f}(x,y)\chi_{\mathbf{D}}(x,y),
\end{eqnarray}
where $\chi_{\mathbf{D}}(x,y):= \left\{
\begin{array}{ll}
1, & (x,y)\in \mathbf{D},\\
[1.5ex]
0,  & {\rm otherwise}.
\end{array}\right.$
How to get the $\mathbf{W}$-bandlimited $\bm{f}$ in the outside of the segment areas is known as
an extrapolation problem.
There is an iteration method to get the $\bm{f}$ in Figure \ref{fig.flow}.
\begin{figure}[!h]
  \centering
    \includegraphics[width=9.0cm]{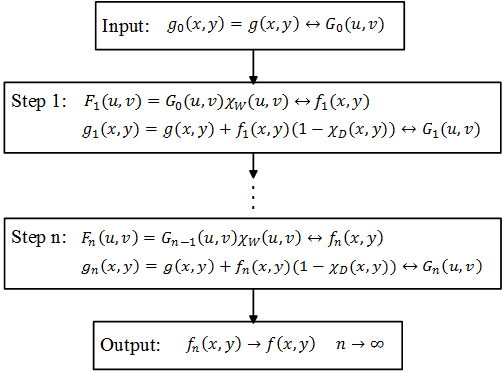}
     \caption{The flow of the iteration.}
 \label{fig.flow}
\end{figure}

In this iteration method, the unknown information for $\bm{f}$ outsider the $\mathbf{D}$ is filled step by step.
The $n$-th step is
\begin{eqnarray*}
\F(\bm{f}_n)(u,v)=\F(\bm{g}_{n-1})(u,v)\chi_{\mathbf{ W}}(u,v)\leftrightarrow
\bm{f}_n(x,y)=\bm{g}_{n-1}(x,y)\ast \left(\frac{\sin Wx}{\pi x}\frac{\sin Wy}{\pi y}\right),
\end{eqnarray*}
where
$
\bm{g}_{n-1}(x,y)=\bm{g}(x,y)+\bm{f}_{n-1}(x,y)(1-\chi_{\mathbf{D}}(x,y))= \left\{
\begin{array}{ll}
\bm{g}(x,y), & (x,y)\in \mathbf{D},\\
[1.5ex]
\bm{f}_{n-1}(x,y),  & {\rm otherwise}.
\end{array}\right.
$
We also show the first step for the scalar part  of a $\H$-valued signal $\bm{f}(x,y)$ in Figure \ref{fig.step1}.
Clearly, there is some new information joined in $\bm{g}(x,y)$ in the first step of the iteration process.
Specifically, $\bm{g}(x,y)$ becomes  $\bm{g}_1(x,y)$  in  Figure \ref{fig.step1}.
\begin{figure}
  \centering
    \includegraphics[width=6.0cm]{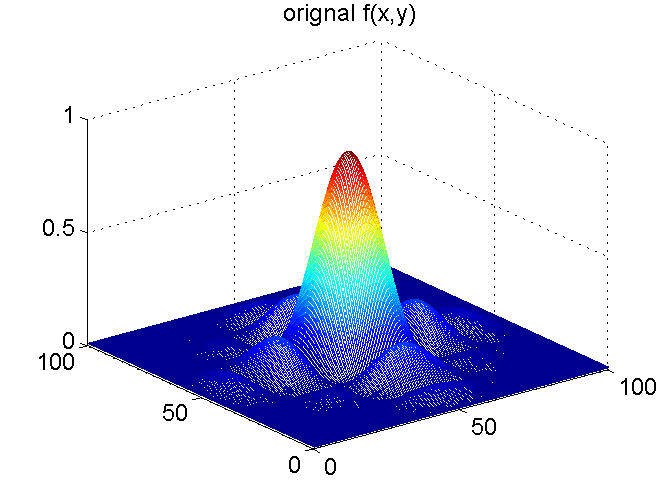}
    \includegraphics[width=6.0cm]{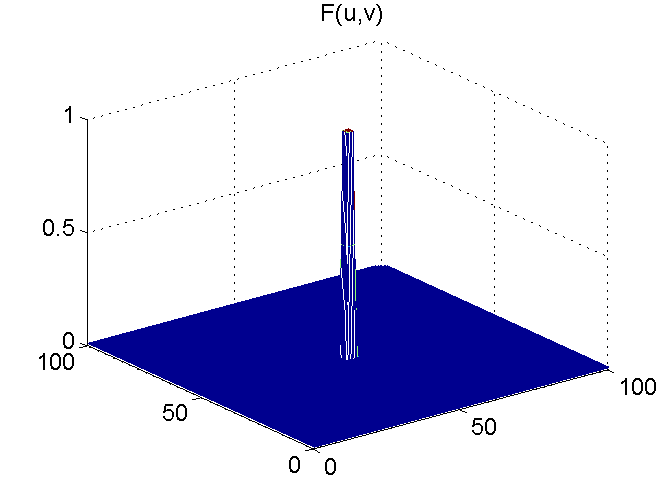}
    \par
    \includegraphics[width=6.0cm]{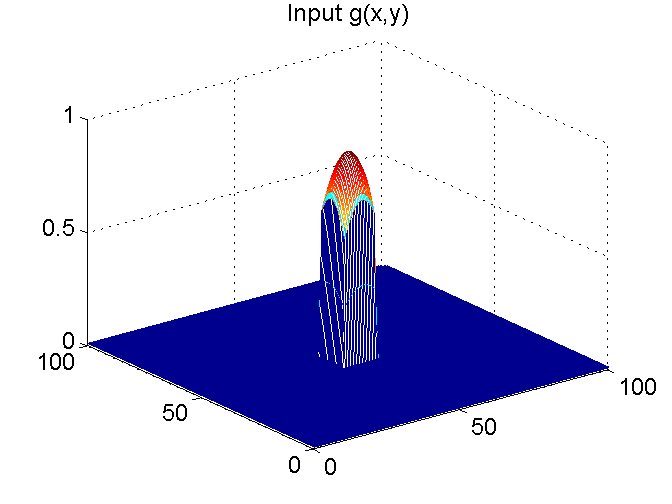}
    \includegraphics[width=6.0cm]{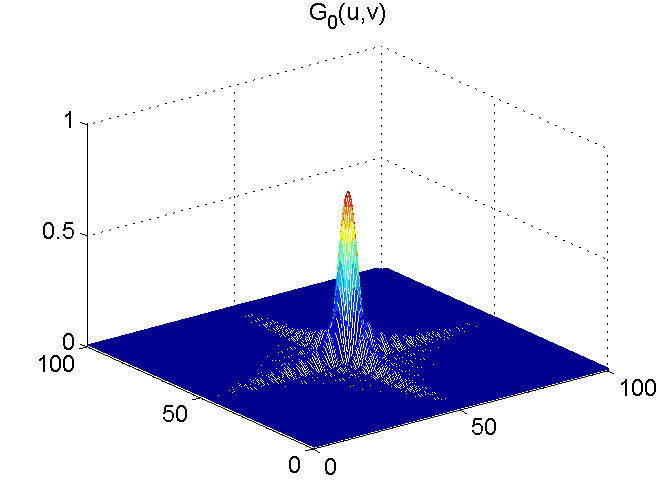}
    \par
    \includegraphics[width=6.0cm]{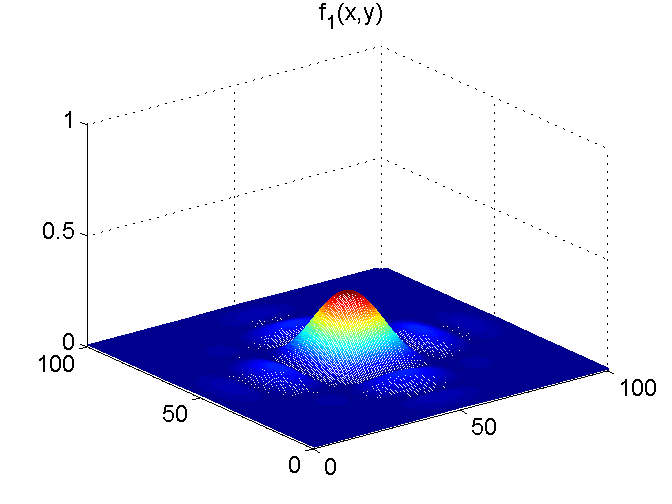}
    \includegraphics[width=6.0cm]{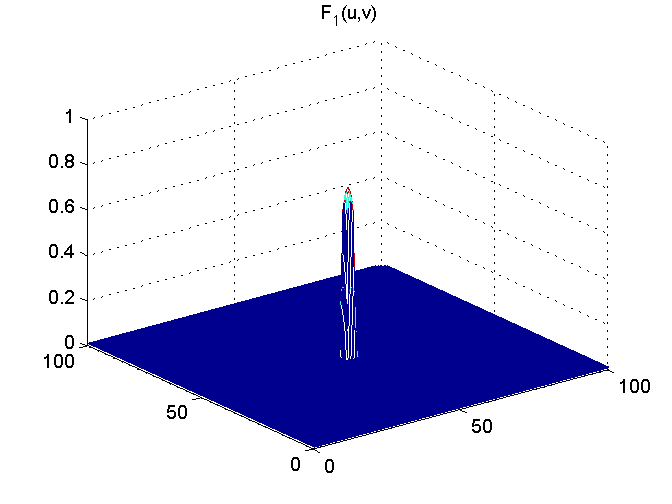}
   \par
    \includegraphics[width=6.0cm]{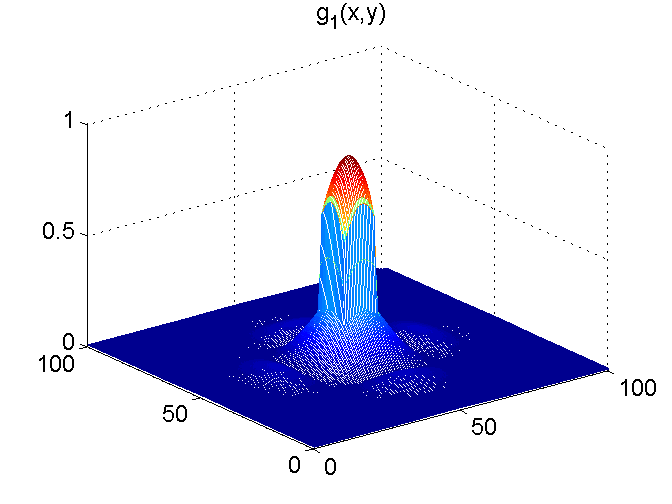}
    \includegraphics[width=6.0cm]{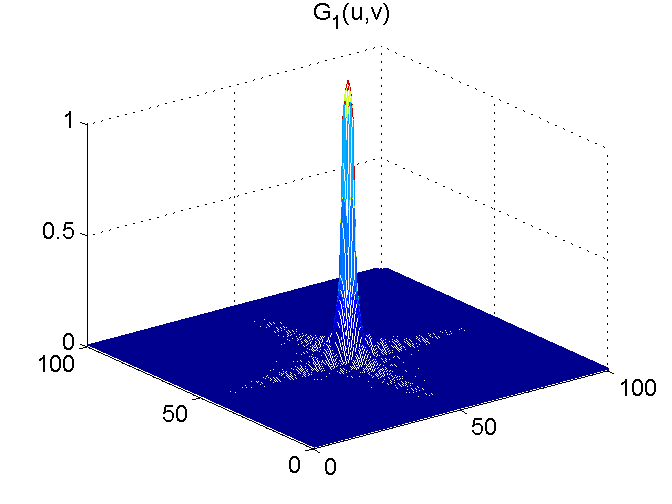}
  \caption{The first step of the scalar part of  an input $\bm{g}(x,y)$.}
 \label{fig.step1}
\end{figure}

Now we will show that the iteration method is effective, i.e., $\bm{f}_n\rightarrow \bm{f}$, $n\rightarrow\infty$.
In order to get this result, we first show the following lemma.
\begin{Lem}\label{Lem:fandfn}
The function of the $n$-th iteration is given by
\begin{eqnarray}
\bm{f}_n(x,y)=\bm{f}(x,y)-\sum_{j=0}^{\infty}a_j(1-\lambda_j)^n\bm{\psi}_j(x,y),
\end{eqnarray}
where $\lambda_j$ are the eigenvalues of QPSWFs $\bm{\psi}_j(x,y)$.
\end{Lem}
\begin{proof}
As we have known, for any $\mathbf{W}$-bandlimited $\H$-valued signal $\bm{f}$ can be expanded into a series
$\bm{f}(x,y)=\sum_{j=0}^{\infty}a_j\bm{\psi}_n(x,y).$
Without loss of generality, we suppose $\bm{f}(x,y)=\bm{\psi}_m(x,y)$.
We shall show that $$\bm{f}_n(x,y)=(1-(1-\lambda_m)^n)\bm{\psi}_m(x,y)=:C_n\bm{\psi}_m(x,y).$$
For $n=1$, $\bm{f}_1(x,y)=(1-(1-\lambda_m)^1)\bm{\psi}_m(x,y)$ is true. Suppose that is true for $n=k$,
then we must show that is true for $k+1$.
Since
\begin{eqnarray*}
\bm{g}_{k}(x,y)&=&\bm{g}(x,y)+\bm{f}_{k}(x,y)(1-\chi_{\mathbf{D}}(x,y))\\
&=&\bm{\psi}_m(x,y)\chi_{\mathbf{D}}(x,y)+ C_k\bm{\psi}_m(x,y)(1-\chi_{\mathbf{D}}(x,y)),
\end{eqnarray*}
then
\begin{eqnarray*}
\bm{f}_{k+1}(x,y)&=&\bm{g}_{k}(x,y)\ast \left(\frac{\sin Wx}{\pi x}\frac{\sin Wy}{\pi y}\right)\\
&=&\left(C_k\bm{\psi}_m(x,y)+ (1-C_k)\bm{\psi}_m(x,y)\chi_{\mathbf{D}}(x,y)\right)\ast \left(\frac{\sin Wx}{\pi x}\frac{\sin Wy}{\pi y}\right).
\end{eqnarray*}
From the low-pass filtering form and  all-pass filtering form of QPSWFs, we have
\begin{eqnarray*}
\bm{f}_{k+1}(x,y)=C_k\bm{\psi}_m(x,y)+(1-C_k)\lambda_m\bm{\psi}_m(x,y)=C_{k+1}\bm{\psi}_m(x,y).
\end{eqnarray*}
Here, we get an  iteration equation about $C_k$ and $C_{k+1}$, i.e.,  $$C_{k+1}=C_k+(1-C_k)\lambda_m.$$
As we know $C_1=\lambda_m$, then directly computation shows that  $C_{k+1}=1-(1-\lambda_m)^{k+1}$.
That means $\bm{f}_n(x,y)=C_n\bm{\psi}_m(x,y)=(1-(1-\lambda_m)^n)\bm{\psi}_m(x,y)$ for $\bm{f}(x,y)=\bm{\psi}_m(x,y)$.
Applying this results to $\bm{f}(x,y)=\sum_{j=0}^{\infty}a_j\bm{\psi}_n(x,y)$, we conclude that
\begin{eqnarray*}
\bm{f}_n(x,y)=\sum_{j=0}^{\infty} (1-(1-\lambda_j)^n)a_j\bm{\psi}_j(x,y)=\bm{f}(x,y)-\sum_{j=0}^{\infty}a_j(1-\lambda_j)^n\bm{\psi}_j(x,y).
\end{eqnarray*}
\end{proof}

Since we have know the relationship between $\bm{f}_n$ and $\bm{f}$ from Lemma \ref{Lem:fandfn},
to show $\bm{f}_n\rightarrow \bm{f}$, with $n\rightarrow\infty$ can also be rewritten as
$\bm{f}-\bm{f}_n\rightarrow 0$, with $n\rightarrow\infty$.
Here we denote the error $e_n(x,y)$ between $\bm{f}_n$ and $\bm{f}$ as
\begin{eqnarray}
e_n(x,y):=\bm{f}(x,y)-\bm{f}_n(x,y)=\sum_{j=0}^{\infty}a_j(1-\lambda_j)^n\bm{\psi}_j(x,y).
\end{eqnarray}
As the $\bm{\psi}_j(x,y)$ are orthogonal in $\R^2$, then the energy $E_n$ of error $e_n$ is
\begin{eqnarray*}
E_n&:=&\int_{\R^2}e_n(x,y)e_n^*(x,y)dxdy
=\int_{\R^2}  \left|\sum_{j=0}^{\infty}a_j(1-\lambda_j)^n\bm{\psi}_j(x,y)\right|^2dxdy\\
&=& \sum_{j=0}^{\infty}a_j^2(1-\lambda_j)^{2n}.
\end{eqnarray*}
Since $E= \sum_{j=0}^{\infty}a_j^2<\infty$, then
for any $\varepsilon >0$, there exists $N$, such that
$\sum_{j>N}a_j^2<\varepsilon$.
On the other hand, $\lambda_i$ are monotonically decreasing as $j\rightarrow 0$. Then
$1-\lambda_j\leq 1-\lambda_N$, for $k\leq N$.
Then we have
\begin{eqnarray*}
E_n&=&\left(\sum_{j=0}^{N}+\sum_{j=N+1}^{\infty}\right)a_j^2(1-\lambda_j)^{2n}
\leq(1-\lambda_N)^{2n} \sum_{j=0}^{N}a_j^2+\sum_{j=N+1}^{\infty}a_j^2\\
&\leq&
(1-\lambda_N)^{2n} E+\varepsilon.
\end{eqnarray*}
That means $E_n\rightarrow 0$ with $n\rightarrow \infty$, because $0<1-\lambda_N<1$.

Another aspect to find is that $e_n(x,y)$ can also be obtained as follows
\begin{eqnarray*}
e_n(x,y)=\frac{1}{2\pi}\int_{\mathbf{W}} e^{\mathbf{i}ux} \left(\F(\bm{f})(u,v)-\F(\bm{f}_n)(u,v)\right)e^{\mathbf{j}vy} dudv,
\end{eqnarray*}
where $\F(\bm{f})$ and $\F(\bm{f}_n)$ are the QFT of $\bm{f}$ and $\bm{f}_n$, respectively. Then we have
\begin{eqnarray*}
|e_n(x,y)|^2&=&\frac{1}{4\pi^2}\left|\int_{\mathbf{W}} e^{\mathbf{i}ux} \left(\F(\bm{f})(u,v)-\F(\bm{f}_n)(u,v)\right)e^{\mathbf{j}vy} dudv\right|^2\\
&\leq&\frac{1}{4\pi^2}\int_{\mathbf{W}} \left|e^{\mathbf{i}ux}\right|^2dudv
\int_{\mathbf{W}} \left|\left(\F(\bm{f})(u,v)-\F(\bm{f}_n)(u,v)\right)e^{\mathbf{j}vy}\right|^2 dudv\\
&=&\frac{W}{\pi^2}E_n.
\end{eqnarray*}
That means
$0\leq|e_n(x,y)|\leq\sqrt{\frac{WE_n}{\pi^2}}\rightarrow 0$, with $n\rightarrow\infty$.
Then we have shown that $\bm{f}-\bm{f}_n\rightarrow 0$, with $n\rightarrow\infty$.

\section{Conclusion}
\label{S7}
In the present paper, we develop the definitions of QPSWFs.
The various properties of QPSWFs such as orthogonality, double orthogonality  are established.
Using the new QPSWFs, we established the
energy distribution for hypercomplex signal (special quaternionic signal) in the QFT domain.
In particular, if a finite energy quaternionic signal is given,
the possible proportions of its energy in a finite time-domain and a finite frequency-domain are found,
as well as the signals which do the best job of simultaneous time and frequency concentration.
It is shown that the QPSWFs will reveal the characteristics of the extreme energy distribution of the original quaternionic signal
and its QFT in different regions.
We also apply the QPSWFs to the bandlimited signals extrapolation problem.
Further investigations on this topic under different linear integral transformation
are now under investigation and will be reported in a forthcoming paper.

\section*{Acknowledgment}
\label{S7}
The first and second authors acknowledge financial support from the National Natural
Science Funds No. 11401606 and University of Macau No. MYRG2015-00058-L2-FST and the Macao Science and Technology
Development Fund FDCT/099/2012/A3. The third author acknowledges financial support from the Asociaci\'on Mexicana de Cultura, A. C.

\section*{Appendix A}
This part proof the Theorem \ref{timezero}.
\begin{proof}
As we know if $\xi =0$, then $\eta_Q \neq 1$ by Lemma \ref{Lementire}.
Although $\eta_Q$ cannot get the value of $1$,
in the following we find the underlying signal for which $\eta_Q$ is infinitely close to $1$.

Let $\mathcal{A} :=\left\{\bm{g}\in \mathcal {L}^2(\R^2,\H) :
\|\bm{g}\|_{\mathcal{L}^2}=1, \,
\|D_{\mathbf{T}}\bm{g}\|_{\mathcal{L}^2} = \xi \right\}$ be a  given function class.
Here, we just consider the condition of $\xi=0$.
Construct a signal $\bm{g}^*(x,y)$  defined as follows
\begin{eqnarray}\label{closeone}
\bm{g}^*(x,y) = \frac{\bm{\psi}_{n} (x,y)-D_{\mathbf{T}}\bm{\psi}_n (x,y)}{\sqrt{1-\lambda_{n}}},
\end{eqnarray}
where $\lambda_{n}$ is the $(n+1)$th eigenvalue of Eq. (\ref{Q2}) and
$\bm{\psi}_{n}$ its corresponding eigenfunction. Direct computations show that
\begin{eqnarray*}
&& \|\bm{g}^*\|_{\mathcal{L}^2}^2 \\
&=& \frac{1}{1-\lambda_{n}} \int_{\R^2} \Big(\bm{\psi}_{n} (x,y) -
D_{\mathbf{T}}\bm{\psi}_n (x,y)\Big)\overline{ \Big( \bm{\psi}_{n} (x,y) - D_{\mathbf{T}}\bm{\psi}_n (x,y) \Big)}dxdy \\
&=& \frac{1}{1-\lambda_{n}} \Big( \int_{\R^2}\bm{\psi}_{n} (x,y)
\overline{\bm{\psi}_{n} (x,y)}dxdy- \int_{\R^2}\bm{\psi}_{n} (x,y) \overline{D_{\mathbf{T}}\bm{\psi}_n (x,y)} dxdy \\
&&  -\int_{\R^2}D_{\mathbf{T}}\bm{\psi}_n (x,y)
\overline{\bm{\psi}_{n} (x,y)}dxdy+ \int_{\R^2}D_{\mathbf{T}}\bm{\psi}_n (x,y) \overline{D_{\mathbf{T}}\bm{\psi}_n (x,y)}dxdy \Big) \\
&=& \frac{1-2\lambda_{n} +\lambda_{n}} {1-\lambda_{n}}= 1,
\end{eqnarray*}
and $\xi =\|D_{\mathbf{T}}\bm{g}^*\|_{\mathcal{L}^2} = 0 $.
Hence $\bm{g}^*\in  \mathcal{A}$.

Now we compute  $\|B_{\mathbf{W}}\bm{g}^* \|_{\mathcal{L}^2}^2$.
Since $\bm{\psi}_{n} (x,y) \in \mathcal {B}_{\mathbf{W}}$,  it follows that
\begin{eqnarray}
B_{\mathbf{W}}\bm{g}^*(x,y) =\frac{\bm{\psi}_{n} (x,y)
-B_{\mathbf{W}}D_{\mathbf{T}} \bm{\psi}_n (x,y)} {\sqrt{1-\lambda_{n}}}.
\end{eqnarray}
 Here, the QFT of $B_{\mathbf{W}}\bm{g}^*(x,y)$ is given as follows
\begin{eqnarray}
\mathcal {F}(B_{\mathbf{W}}\bm{g}^*)
=
D_\mathbf{W}\Bigl( \mathcal {F}(g^*_0)+\mathbf{i}\mathcal {F}(g^*_1)+
\mathcal {F}(g^*_2)\mathbf{j}+\mathbf{i}\mathcal {F}(g^*_3)\mathbf{j}\Bigr),
\end{eqnarray}
and we want to compute
\begin{eqnarray}\label{BWQ}
 \|\mathcal {F}(B_{\mathbf{W}}\bm{g}^*)\|_Q^2=\int_{\mathbf{W}}|\mathcal {F}(B_{\mathbf{W}}\bm{g}^*)|_Q^2 dudv
=\sum_{i=0}^3\int_{\mathbf{W}}\Bigl(|\mathcal {F}(B_{\mathbf{W}}g^*_i)|^2\Bigr)dudv,
\end{eqnarray}
where for $i=0,1,2,3,$
\begin{eqnarray}
B_{\mathbf{W}}g^*_i(x,y)=\frac{\psi_{n,i}(x,y)-
B_{\mathbf{W}}D_{\mathbf{T}}\psi_{n,i} (x,y)}{2(\sqrt{1-\lambda_{n}})},
\end{eqnarray}
and the QFT of
$B_{\mathbf{W}}g^*_i(x,y)$ satisfies that
$ \mathcal {F}\Big(B_{\mathbf{W}}g^*_i\Big)=D_\mathbf{W}\Big(\mathcal {F}(g^*_i)\Big)$.
Hence, in order to compute $ \|\mathcal {F}(B_{\mathbf{W}}\bm{g}^*)\|_Q^2$ in Eq. (\ref{BWQ}),
we need to compute each term of
$\int_{\R^2}|B_{\mathbf{W}}g^*_i|^2dxdy$, for $i=0,1,2,3$.

Since for real-valued signals $B_{\mathbf{W}}g^*_i(x,y)$ and their QFT $\mathcal {F}(B_{\mathbf{W}}g^*_i)$,
Parseval's theorem holds \cite{CKL2015}, then for all $i=0,1,2,3$
\begin{eqnarray}
\int_{\R^2}|B_{\mathbf{W}}g^*_i|^2dxdy=\int_{\mathbf{W}}|\mathcal {F}(B_{\mathbf{W}}g^*_i)|^2 dudv
\end{eqnarray}

Having in mind that for $i=0,1,2,3$
\begin{eqnarray*}
B_{\mathbf{W}}D_{\mathbf{T}}\psi_{n,i}(x,y)
&=&\frac{1}{(2\pi) ^2}\frac{1}{\mathbf{i}^n}
\int_{\mathbf{W}}\int_{ \mathbf{T}}
e^{-\mathbf{i}ux}e^{\mathbf{i}us}
\psi_{n,i}(s ,t )e^{\mathbf{j}vt}e^{-\mathbf{j}vy}
dsdtdudv\frac{1}{\mathbf{j}^n}\\
&=&\mathbf{i}^n\Big(\int_{ \mathbf{T}}
\frac{\sin W(x-s)}{\pi(x-s)}\psi_{n,i}(s,t)
\frac{\sin W(y-t)}{\pi(y-t)}
dsdt\Big)\mathbf{j}^n.
\end{eqnarray*}

straightforward computations show that
\begin{eqnarray} \label{EqBWintegral}
&&\int_{\R^2}|B_{\mathbf{W}}g^*_i|^2dxdy \\  \nonumber
&=&\frac{1}{4(1-\lambda_{n})}\int_{\R^2}
\Big(\psi_{n,i} (x,y)
-B_{\mathbf{W}}D_{\mathbf{T}}\psi_{n,i}(x,y)\Big)\overline{\Big(\psi_{n,i} (x,y)
-B_{\mathbf{W}}D_{\mathbf{T}}\psi_{n,i} (x,y)\Big)}dxdy\\ \nonumber
&=&\frac{1}{4(1-\lambda_{n})}
\Big(\int_{\R^2}\psi_{n,i} (x,y)
\overline{\psi_{n,i} (x,y)}dxdy-\int_{\R^2}\psi_{n,i} (x,y)
\overline{B_{\mathbf{W}}D_{\mathbf{T}}\psi_{n,i} (x,y)}dxdy\\ \nonumber
&&-\int_{\R^2}B_{\mathbf{W}}D_{\mathbf{T}}\psi_{n,i} (x,y)
\overline{\psi_{n,i} (x,y)}dxdy+\int_{\R^2}B_{\mathbf{W}}D_{\mathbf{T}}\psi_{n,i} (x,y)
\overline{B_{\mathbf{W}}D_{\mathbf{T}}\psi_{n,i} (x,y)}dxdy\Big)\\ \nonumber
&=& \left\{
\begin{array}{ll}
\frac{1-2\lambda_{n} +\lambda_{n} ^2}
{4(1-\lambda_{n})}, & n=2k, ~~~~~~k=1,2,\cdots\\
[1.5ex]
\frac{1+\lambda_{n} ^2}
{4(1-\lambda_{n})},  &n=2k-1,~k=1,2,\cdots.
\end{array}\right. \nonumber
\end{eqnarray}
Here, the second equality for Eq. (\ref{EqBWintegral}) have four terms integral.
The first one is the orthogonality of QPSWFs.
Computations of the second integral show that
\begin{eqnarray*}
&& \int_{\R^2}
\psi_{n,i}(x,y)
\overline{B_{\mathbf{W}}D_{\mathbf{T}}\psi_{n,i}(x,y)}dxdy\\
&=&\int_{\R^2}\int_{ \mathbf{T}}\psi_{n,i}(x ,y )
\frac{\sin W(s-x)}{\pi (s-x)}(-\mathbf{j})^n\overline{\psi_{n,i}(s, t )}(-\mathbf{i})^n
\frac{\sin W(t-y)}{\pi(t-y)}dsdtdxdy\\
&=&\lambda_{n}\int_{\R^2}\psi_{n,i}(x ,y )(-\mathbf{j})^n
\overline{\psi_{n,i}(x ,y )}(-\mathbf{i})^ndxdy
=\lambda_{n}(-\mathbf{j})^n(-\mathbf{i})^n.
\end{eqnarray*}
Similarly, we can immediately obtain the third integral that
\begin{eqnarray}
\int_{\R^2}B_{\mathbf{W}}D_{\mathbf{T}}\psi_{n,i}(x,y)
\overline{\psi_{n,i}(x,y)}dxdy=
\lambda_{n}\mathbf{i}^n\mathbf{j}^n.
\end{eqnarray}
Then  for  the two cross-terms for the second equality in Eq. (\ref{EqBWintegral}) have
\begin{eqnarray*}
 &&\int_{\R^2}\psi_{n,i}(x,y)\overline{B_{\mathbf{W}}D_{\mathbf{T}}\psi_{n,i}(x,y)}dxdy
+\int_{\R^2}B_{\mathbf{W}}D_{\mathbf{T}}\psi_{n,i}(x,y)\overline{\psi_{n,i}(x,y)}dxdy\\
&=&\lambda_{n}(-\mathbf{j})^n(-\mathbf{i})^n + \lambda_{n}\mathbf{i}^n\mathbf{j}^n
=\left\{
\begin{array}{ll}
2\lambda_{n}, & n=2k, ~k=1,2,\cdots\\
[1.5ex]
0,  &n=2k-1,~k=1,2,\cdots
\end{array}\right.
\end{eqnarray*}

The last integral for the second equality in Eq. (\ref{EqBWintegral}) used the Lemma \ref{expart} to compute that
\begin{eqnarray*}
&&\int_{\R^2}
B_{\mathbf{W}}D_{\mathbf{T}}\psi_{n,i}(x,y)
\overline{B_{\mathbf{W}}D_{\mathbf{T}}\psi_{n,i}(x,y)}dxdy\\
&=&\int_{\R^2}
\Big(\mathbf{i}^n\int_{ \mathbf{T}}
\frac{\sin W(x-s)}{\pi(x-s)}
\psi_{n,i}(s ,t )\frac{\sin W(y-t)}{\pi(y-t)}
dsdt\mathbf{j}^n\Big)\\
&&\Big((-\mathbf{j})^n\int_{ \mathbf{T}}
\frac{\sin W(x-z)}{\pi(x-z)}
\overline{\psi_{n,i}(z ,w )}\frac{\sin W(y-w)}{\pi(y-w)}
dzdw(-\mathbf{i})^n\Big)dxdy\\
&=&\mathbf{i}^n\Big(\int_{ \mathbf{T}}\int_{ \mathbf{T}}
\psi_{n,i}(s ,t )
\overline{\psi_{n,i}(z ,w )}
\frac{\sin W(s-z)}{\pi(s-z)}\frac{\sin W(t-w)}{\pi(t-w)}
dsdtdzdw\Big)(-\mathbf{i})^n\\
&=&\lambda_{n}\mathbf{i}^n\Big(
\int_{ \mathbf{T}}\psi_{n,i}(z ,w )
\overline{\psi_{n,i}(z ,w )}dzdw\Big)(-\mathbf{i})^n
=\lambda_{n} ^2.
\end{eqnarray*}
Using the above results of Eq. (\ref{EqBWintegral}), there holds
\begin{eqnarray*}
\eta_Q^2&=&\|\mathcal {F}(B_{\mathbf{W}}\bm{g}^*)\|_Q^2
=\sum_{i=0}^3\int_{\mathbf{W}}|\mathcal {F}(g^*_i)|^2 dudv\\
&=&\sum_{i=0}^3\int_{\R^2}|B_{\mathbf{W}}g^*_i|^2dxdy=
\left\{
\begin{array}{ll}
\frac{1-2\lambda_{n} +\lambda_{n} ^2}
{1-\lambda_{n}}, & n=2k, ~~~~~~k=1,2,\cdots\\
[1.5ex]
\frac{1+\lambda_{n} ^2}
{1-\lambda_{n}},  &n=2k-1,~ k=1,2,\cdots.
\end{array}\right.
\end{eqnarray*}
Therefore $\bm{g}^*\in  \mathcal{A}$ and $\eta_Q = \|B_{\mathbf{W}}\bm{g}^* \|_Q=\sqrt{1-\lambda_{n}}$, $\forall ~n=2k,~k=1,2,\cdots$.
As $\{\lambda_{n} \}_{n=0}^{\infty}$ is monotone decreasing in the interval $(0,1)$,
$\lambda_{n} $ can be arbitrarily close to $0$. Thus  $\eta_Q$ can be arbitrarily close to $1$.

On the other hand, it is clear that when both energy ratios $\xi$ and $\eta_Q$ are equal to $0$,
the underlying quaternionic signal must be the identically zero signal.
Nevertheless, we are able to finding a signal that is not identically zero as such $\eta_Q$ is infinitely close to $0$.
For any $r \in \R$ and the $\bm{g}^*(x,y)$ in the Eq. (\ref{closeone}),
we note the QFT of $\bm{g}^*(x,y)$ as $\mathcal {F}(\bm{g}^*)$,
where $\mathcal {F}(\bm{g}^*)=\mathcal {F}(g^*_0)+\mathbf{i}\mathcal {F}(g^*_1)+
\mathcal {F}(g^*_2)\mathbf{j}+\mathbf{i}\mathcal {F}(g^*_3)\mathbf{j}$.
We want to find a new function $\bm{f}(x,y)$ such that the QFT of $\bm{f}(x,y)$ satisfy that
\begin{eqnarray}
\mathcal {F}(\bm{f})(u,v)=\mathcal {F}(\bm{g}^*)(u-r,v).
\end{eqnarray}
If there exist that $\bm{f}(x,y)$, then
\begin{eqnarray*}
\eta_Q ^2&=&\|\mathcal {F}(B_{\mathbf{W}}\bm{f})\|_Q^2=\int_{ \mathbf{W}} \sum_{i=0}^{3}
|\mathcal {F}(f_i)(u,v)|^2 dudv\\
&=&\int_{ \mathbf{W}} \sum_{i=0}^{3}
|\mathcal {F}(g_i)(u-r,v)|^2 dudv=\int_{-W -r}^{W -r}\int_{-W}^{W}
 \sum_{i=0}^{3}
|\mathcal {F}(g_i)(u,v)|^2 dudv.
\end{eqnarray*}
Here, $\eta_Q $ is continuous over $r $ for fixed $\mathbf{W}$.
For $\mathcal {F} (\bm{g}^*)(u,v) \in \mathcal {L}^2(\R^2;\H)$,
 $\eta_Q $ approaches zero as $r$ approaches infinity.
 Thus,  $\eta_Q $ can be  arbitrarily close to $0$.

Now we need  to present the formula of that $\bm{f}(x,y)$ and check whether that $\bm{f}(x,y)$ belong to $\mathcal{A}$.
Since $\bm{f}(x,y)$ satisfy that $\mathcal {F}(\bm{f})(u,v)=\mathcal {F}(\bm{g}^*)(u-r,v)$, it follows that
 \begin{eqnarray*}
\bm{f}(x,y)
 &=&\frac{1}{2\pi}\int_{\R^2}e^{\mathbf{i}ux}\mathcal {F}(\bm{f})(u,v)e^{\mathbf{j}vy}dudv
=\frac{1}{2\pi}\int_{\R^2}e^{\mathbf{i}ux}\mathcal {F}(\bm{g}^*)(u-r,v)e^{\mathbf{j}vy}dudv\\
 &=&\frac{1}{2\pi}\int_{\R^2}e^{\mathbf{i}(u+r)x}\mathcal {F}(\bm{g}^*)(u,v)e^{\mathbf{j}vy}dudv
=e^{\mathbf{i}rx}\frac{1}{2\pi}\int_{\R^2}e^{\mathbf{i}ux}\mathcal {F}(\bm{g}^*)(u,v)e^{\mathbf{j}vy}dudv\\
 &=&e^{\mathbf{i}rx}\bm{g}^*(x,y).
  \end{eqnarray*}
It is easy check that
 \begin{eqnarray*}
\|\bm{f}\|_{\mathcal{L}^2}^2
=\int_{\R^2} |e^{\mathbf{i}rx}\bm{g}^*(x,y)|^2 dxdy
=\|\bm{g}^* \|_{\mathcal{L}^2}^2=1,
 \end{eqnarray*}
and
\begin{eqnarray*}
\xi
=\int_{\mathbf{T}} |e^{\mathbf{i}rx}\bm{g}^*(x,y)|^2 dxdy
=\int_{\mathbf{T}} |\bm{g}^*(x,y)|^2 dxdy=0.
\end{eqnarray*}
Thus, $ \bm{f}(x,y)\in \mathcal{A}$ and $\eta_Q $ can be  arbitrarily close to $0$ when $r\rightarrow \infty$.
That is, $0\leq\eta_Q <1$.
This completes the proof.
\end{proof}

\section*{Appendix B}
This part proof the Theorem \ref{ThUP}.
\begin{proof}
Let us recall the function class
\begin{eqnarray*}
\mathcal{A} = \left\{\bm{g}\in \mathcal {L}^2(\R^2,\H) :
\|\bm{g}\|_{\mathcal{L}^2}=1, \, \|D_{\mathbf{T}}\bm{g}\|_{\mathcal{L}^2} = \xi \right\}.
\end{eqnarray*}
Let $\bm{g}(x,y)\in \mathcal{A}$ and $\sqrt{\lambda _{0}}\leq \xi < 1$.
Now, take its projections in $\mathcal {D}_\mathbf{T}$ and $\mathcal {B}_\mathbf{W}$, respectively.
We  can decompose $\bm{g}$ as follows
\begin{eqnarray} \label{Decomposition}
\bm{g} = \lambda D_{\mathbf{T}}\bm{g} +\mu B_{\mathbf{W}}\bm{g} +\bm{g}^*,
\end{eqnarray}
where $\lambda$, $\mu \in \R$, $< \bm{g}^* ,D_{\mathbf{T}}\bm{g} >_{\mathcal{L}^2}
=\bm{0}$ and $< \bm{g}^* ,B_{\mathbf{W}}\bm{g} >_{\mathcal{L}^2} = \bm{0}$.
The $g(x,y)\in \mathcal{A}_{\xi}$ can also decomposed by
 \begin{eqnarray}
\bm{g} =g_0+\mathbf{i}g_1+g_2\mathbf{j}+\mathbf{i}g_3\mathbf{j}.
\end{eqnarray}
Then we need to consider $i=0,1,2,3$,
\begin{eqnarray}
g_i = \lambda D_{\mathbf{T}}g_i +\mu B_{\mathbf{W}}g_i +g^*_i,
\end{eqnarray}
Now, we compute the scalar inner product Eq. (\ref{Real_Inner_Product}) of the decomposition Eq. (\ref{Decomposition}),
respectively, with $g_i$, $D_{\mathbf{T}}g_i$, $B_{\mathbf{W}}g_i $ and $g^*_i $. Direct computations show that
\begin{eqnarray*}
\|g_i\|_{\mathcal{L}^2}^2
&=&  \lambda \|D_{\mathbf{T}}g_i\|_{\mathcal{L}^2}^2
+\mu \|B_{\mathbf{W}}g_i\|_{\mathcal{L}^2}^2+< g_i ,g^*_i >,\\
\|D_{\mathbf{T}}g_i\|_{\mathcal{L}^2}^2
&=&\lambda \|D_{\mathbf{T}}g_i\|_{\mathcal{L}^2}^2+\mu < B_{\mathbf{W}}g_i ,D_{\mathbf{T}}g_i>_0,\\
\|B_{\mathbf{W}}g_i\|_{\mathcal{L}^2}^2
&=&  \lambda < D_{\mathbf{T}}g_i ,B_{\mathbf{W}}g_i>_0+\mu \|B_{\mathbf{W}}g_i\|_{\mathcal{L}^2}^2,\\
< g_i,g^*_i >
&=& <g^*_i ,g^*_i >.
\end{eqnarray*}
For every $\H$-valued function $\bm{f}$, the equation
$\|\bm{f}\|_{\mathcal{L}^2}^2 =\|\mathcal {F}(\bm{f}) \|_Q ^2$ holds.
Then we can get the following results by integrating the equations above $i=0,1,2,3$
\begin{eqnarray*}
\|\mathcal {F}(\bm{g}) \|_Q ^2
&=&  \lambda \|\mathcal {F}(D_{\mathbf{T}}\bm{g})\|_Q ^2+
\mu \|\mathcal {F}(B_{\mathbf{W}}\bm{g})\|_Q ^2+\|\mathcal {F}(\bm{g}^*)\|_Q ^2,\\
\|\mathcal {F}(D_{\mathbf{T}}\bm{g})\|_Q ^2
&=&  \lambda \|\mathcal {F}(D_{\mathbf{T}}\bm{g})\|_Q ^2+
\mu \sum_{i=0}^3< B_{\mathbf{W}}g_i ,D_{\mathbf{T}}g_i>_0,\\
\|\mathcal {F}(B_{\mathbf{W}}\bm{g})\|_Q ^2
&=&  \lambda \sum_{i=0}^3< D_{\mathbf{T}}g_i ,B_{\mathbf{W}}g_i>_0 +
\mu \|\mathcal {F}(B_{\mathbf{W}}\bm{g})\|_Q ^2. \\
\end{eqnarray*}
Substituting the facts $\xi=\|\mathcal {F}(D_{\mathbf{T}}\bm{g})\|_Q$, $\eta_Q=\|\mathcal {F}(B_{\mathbf{W}}\bm{g})\|_Q$
and $\|\mathcal {F}(\bm{g}^*)\|_Q=\|\bm{g}^*\|_{\mathcal{L}^2}$ into the equations above, it follows that
\begin{eqnarray*}
1 &=&  \lambda \xi^2+\mu \eta_Q^2+\|\bm{g}^*\|_{\mathcal{L}^2}^2,\\
\xi^2 &=& \lambda \xi^2+\mu \sum_{i=0}^3< B_{\mathbf{W}}g_i ,D_{\mathbf{T}}g_i>_0,\\
\eta_Q^2 &=&  \lambda \sum_{i=0}^3< D_{\mathbf{T}}g_i ,B_{\mathbf{W}}g_i>_0 +\mu \eta_Q^2,
\end{eqnarray*}
By eliminating $\|\bm{g}^*\|_{\mathcal{L}^2}^2$,
$\lambda $ and $\mu $ from the above equations, we obtain that
\begin{eqnarray}\label{eta}
&&\eta_Q^2-\sum_{i=0}^3< B_{\mathbf{W}}g_i ,D_{\mathbf{T}}g_i>_0-\sum_{i=0}^3< D_{\mathbf{T}}g_i ,B_{\mathbf{W}}g_i>_0\\ \nonumber
&=&~-\xi^2 +\left(1-\|\bm{g}^*\|_{\mathcal{L}^2}^2\right)\cdot \left(1-
 \frac{(\sum_{i=0}^3< B_{\mathbf{W}}g_i ,D_{\mathbf{T}}g_i>_0)
(\sum_{i=0}^3< D_{\mathbf{T}}g_i ,B_{\mathbf{W}}g_i>_0)}{\xi^2\eta_Q^2}\right),
\end{eqnarray}
where $\xi\eta_Q\neq 0$.
For $\sum_{i=0}^3< B_{\mathbf{W}}g_i ,D_{\mathbf{T}}g_i>_0$ and
$\sum_{i=0}^3< D_{\mathbf{T}}g_i ,B_{\mathbf{W}}g_i>_0$ in the above equation, we conclude  that
\begin{eqnarray}
< B_{\mathbf{W}}\bm{g},D_{\mathbf{T}}\bm{g}>_0=\sum_{i=0}^3< B_{\mathbf{W}}g_i ,D_{\mathbf{T}}g_i>_0.
\end{eqnarray}
Since
\begin{eqnarray*}
\bm{g}&=&g_0 +\mathbf{i}g_1+g_2\mathbf{j}+\mathbf{i}g_3\mathbf{j},\\
D_{\mathbf{T}}\bm{g} &=&D_{\mathbf{T}}g_0+\mathbf{i}D_{\mathbf{T}}g_1+
D_{\mathbf{T}}g_2\mathbf{j}+\mathbf{i}D_{\mathbf{T}}g_3\mathbf{j},\\
B_{\mathbf{W}}\bm{g}&=&B_{\mathbf{W}}g_0+\mathbf{i}B_{\mathbf{W}}g_1+
B_{\mathbf{W}}g_2\mathbf{j}+\mathbf{i}B_{\mathbf{W}}g_3\mathbf{j},
\end{eqnarray*}
and
\begin{eqnarray*}
&&< B_{\mathbf{W}}\bm{g},D_{\mathbf{T}}\bm{g}>_0\\
&=& \mathbf{Sc}\left(\int_{\R^2}
B_{\mathbf{W}}\bm{g}(x,y)\overline{D_{\mathbf{T}}\bm{g}(x,y)}dxdy\right)\\
&=&\mathbf{Sc}\Big(\int_{\R^2}
(B_{\mathbf{W}}g_0+\mathbf{i}B_{\mathbf{W}}g_1+
B_{\mathbf{W}}g_2\mathbf{j}+\mathbf{i}B_{\mathbf{W}}g_3\mathbf{j})
\overline{(D_{\mathbf{T}}g_0+\mathbf{i}D_{\mathbf{T}}g_1+
D_{\mathbf{T}}g_2\mathbf{j}+\mathbf{i}D_{\mathbf{T}}g_3\mathbf{j})}dxdy\Big)\\
&=&\int_{\R^2}\Big(B_{\mathbf{W}}g_0D_{\mathbf{T}}g_0+B_{\mathbf{W}}g_1D_{\mathbf{T}}g_1
+B_{\mathbf{W}}g_2D_{\mathbf{T}}g_2+B_{\mathbf{W}}g_3D_{\mathbf{T}}g_3\Big)dxdy\\
&=&\sum_{i=0}^3< B_{\mathbf{W}}g_i ,D_{\mathbf{T}}g_i>_0.
\end{eqnarray*}
Then  the Eq. (\ref{eta}) becomes
\begin{eqnarray}
\eta_Q^2-2< B_{\mathbf{W}}\bm{g},D_{\mathbf{T}}\bm{g}>_0
=-\xi^2 +\Big(1-\|\bm{g}^*\|_{\mathcal{L}^2}^2\Big)
\Big(1-\frac{(< B_{\mathbf{W}}\bm{g},D_{\mathbf{T}}\bm{g}>_0)^2}{\xi^2\eta_Q^2}\Big).
\end{eqnarray}
Let $\arg(D_{\mathbf{T}}\bm{g},B_{\mathbf{W}}\bm{g})
= \arccos\frac{< D_{\mathbf{T}}\bm{g} ,B_{\mathbf{W}}\bm{g}>_0}
{\|D_{\mathbf{T}}\bm{g}\|_{\mathcal{L}^2}
\|B_{\mathbf{W}}\bm{g}\|_{\mathcal{L}^2}}$.
By Theorem \ref{Thangle},  it follows that
$$\arg(D_{\mathbf{T}}\bm{g},B_{\mathbf{W}}\bm{g})\geq \arccos \sqrt{\lambda _{0}}.$$
 We conclude that
\begin{eqnarray*}
\eta_Q^2-2\xi\eta_Q\cos\arg(D_{\mathbf{T}}\bm{g},B_{\mathbf{W}}\bm{g})
&=&-\xi^2 +\Big(1-\|\bm{g}^*\|_{\mathcal{L}^2}^2\Big)
\Big(1-\cos^2\arg(D_{\mathbf{T}}\bm{g},B_{\mathbf{W}}\bm{g})\Big)\\
&\leq& -\xi^2 +\sin^2\arg(D_{\mathbf{T}}\bm{g},B_{\mathbf{W}}\bm{g}).
\end{eqnarray*}
Simplifying the above inequality, we obtain
$
\Big(\eta_Q -\xi\cos \arg(D_{\mathbf{T}}\bm{g},B_{\mathbf{W}}\bm{g})\Big)^2 \leq
 (1- \xi ^2)\sin ^2 \arg(D_{\mathbf{T}}\bm{g},B_{\mathbf{W}}\bm{g}),
$
from which it follows that
$
\eta_Q \leq \cos \Big(\arg(D_{\mathbf{T}}\bm{g},B_{\mathbf{W}}\bm{g}) - \arccos \xi\Big),
$
 and
$
\arg(D_{\mathbf{T}}\bm{g},B_{\mathbf{W}}\bm{g})\geq \arccos \sqrt{\lambda _{0}}.
$
We conclude that
\begin{eqnarray}
\arccos\eta_Q \geq \arccos \sqrt{\lambda _{0}} - \arccos \xi.
\end{eqnarray}
The equality is attained by setting
\begin{eqnarray}
\bm{g}^*(x,y)=p\bm{\psi} _{0}(x,y)+qD_{\mathbf{T}}\bm{\psi} _0(x,y),
\end{eqnarray}
where
$p=\sqrt{\frac{1-\xi ^2}{1-\lambda_{0}}}$, and
$q=\frac{\xi }{\sqrt{\lambda_{0}}}
-\sqrt{\frac{1-\xi ^2}{1-\lambda_{0}}}$.
The proof is completed.
\end{proof}


\section*{References}
\begin{thebibliography}{10}

\bibitem{R1968}
A.~Rihaczek, Signal energy distribution in time and frequency, IEEE
  Transactions on Information Theory 14~(3) (1968) 369--374.
\newblock \href {http://dx.doi.org/10.1109/TIT.1968.1054157}
  {\path{doi:10.1109/TIT.1968.1054157}}.

\bibitem{TP1987}
N.~Tugbay, E.~Panayirci, Energy optimization of band-limited nyquist signals in
  the time domain, IEEE Transactions on Communications 35~(4) (1987) 427--434.
\newblock \href {http://dx.doi.org/10.1109/TCOM.1987.1096794}
  {\path{doi:10.1109/TCOM.1987.1096794}}.

\bibitem{SP1961}
D.~Slepian, H.~O. Pollak, Prolate spheroidal wave functions, fourier analysis,
  and uncertainty--{I}, Bell System Technical Journal 40~(1) (1961) 43--64.
\newblock \href {http://dx.doi.org/10.1002/j.1538-7305.1961.tb03976.x}
  {\path{doi:10.1002/j.1538-7305.1961.tb03976.x}}.

\bibitem{LP1961}
H.~J. Landau, H.~O. Pollak, Prolate spheroidal wave functions, fourier analysis
  and uncertainty--{II}, Bell System Technical Journal 40~(1) (1961) 65--84.
\newblock \href {http://dx.doi.org/10.1002/j.1538-7305.1961.tb03977.x}
  {\path{doi:10.1002/j.1538-7305.1961.tb03977.x}}.

\bibitem{LP1962}
H.~J. Landau, H.~O. Pollak, Prolate spheroidal wave functions, fourier analysis
  and uncertainty--{III}: The dimension of space of essentially time-and
  bandlimited signals, Bell System Technical Journal 41~(4) (1962) 1295--1336.
\newblock \href {http://dx.doi.org/10.1002/j.1538-7305.1962.tb03279.x}
  {\path{doi:10.1002/j.1538-7305.1962.tb03279.x}}.

\bibitem{S1964}
D.~Slepian, Prolate spheroidal wave functions, fourier analysis and
  uncertainty--{IV}: Extensions to many dimensions; generalized prolate
  spheroidal functions, Bell System Technical Journal 43~(6) (1964) 3009--3057.
\newblock \href {http://dx.doi.org/10.1002/j.1538-7305.1964.tb01037.x}
  {\path{doi:10.1002/j.1538-7305.1964.tb01037.x}}.

\bibitem{LW1980}
H.~J. Landau, H.~Widom, Eigenvalue distribution of time and frequency limiting,
  Mathematical Analysis and Applications 77~(2) (1980) 469--481.
\newblock \href {http://dx.doi.org/10.1016/0022-247X(80)90241-3}
  {\path{doi:10.1016/0022-247X(80)90241-3}}.

\bibitem{MC2004}
I.~C. Moorea, M.~Cada, Prolate spheroidal wave functions, an introduction to
  the slepian series and its properties, Applied and Computational Harmonic
  Analysis 16~(3) (2004) 208--230.
\newblock \href {http://dx.doi.org/10.1016/j.acha.2004.03.004}
  {\path{doi:10.1016/j.acha.2004.03.004}}.

\bibitem{KM2008}
I.~C. Moorea, M.~Cada, New efficient methods of computing the prolate
  spheroidal wave functions and their corresponding eigenvalues, Applied and
  Computational Harmonic Analysis 24~(3) (2008) 269--289.
\newblock \href {http://dx.doi.org/10.1016/j.acha.2007.06.004}
  {\path{doi:10.1016/j.acha.2007.06.004}}.

\bibitem{S1967}
S.~Slavyanov, Asymptotic forms for prolate spheroidal functions, USSR
  Computational Mathematics and Mathematical Physics 7~(5) (1967) 50--62.
\newblock \href {http://dx.doi.org/10.1016/0041-5553(67)90093-6}
  {\path{doi:10.1016/0041-5553(67)90093-6}}.

\bibitem{WS2005}
G.~Walter, T.~Soleski, A new friendly method of computing prolate spheroidal
  wave functions and wavelets, Applied and Computational Harmonic Analysis
  19~(3) (2005) 432--443.
\newblock \href {http://dx.doi.org/10.1016/j.acha.2005.04.001}
  {\path{doi:10.1016/j.acha.2005.04.001}}.

\bibitem{OR2012}
Detailed analysis of prolate quadratures and interpolation formulas, Numerical
  Analysis.

\bibitem{K2010}
Uncertainty principles, prolate spheroidal wave functions, and applications,
  Applied and Numerical Harmonic Analysis.

\bibitem{MKZ2013}
K.~K. J.~Morais, Y.~Zhang, Generalized prolate spheroidal wave functions for
  offset linear canonical transform in clifford analysis, Mathematical Methods
  in the Applied Sciences 36~(9) (2013) 1028--1041.
\newblock \href {http://dx.doi.org/10.1002/mma.2657}
  {\path{doi:10.1002/mma.2657}}.

\bibitem{MK2014}
Constructing prolate spheroidal quaternion wave signals on the sphere,
  Mathematical Methods in the Applied Sciences.

\bibitem{Z2007}
A.~I. Zayed, A generalization of the prolate spheroidal wave functions,
  Proceedings of the American mathematical society 135~(7) (2007) 2193--2203.
\newblock \href
  {http://dx.doi.org/http://dx.doi.org/10.1090/S0002-9939-07-08739-4}
  {\path{doi:http://dx.doi.org/10.1090/S0002-9939-07-08739-4}}.

\bibitem{WS2004}
G.~Walter, X.~Shen, Wavelets based on prolate spheroidal wave functions,
  Fourier Analysis and Applications 10~(1) (2004) 1--26.
\newblock \href {http://dx.doi.org/10.1007/s00041-004-8001-7}
  {\path{doi:10.1007/s00041-004-8001-7}}.

\bibitem{Z2014}
Generalized and fractional prolate spheroidal wave functions, Proceedings of
  the 10th International Conference on Sampling Theory and Applications.

\bibitem{PD2005}
S.~Pei, J.~Ding, Generalized prolate spheroidal wave functions for optical
  finite fractional fourier and linear canonical transforms, Optical Society of
  America A 22~(3) (2005) 460--474.
\newblock \href {http://dx.doi.org/10.1364/JOSAA.22.000460}
  {\path{doi:10.1364/JOSAA.22.000460}}.

\bibitem{ZRMT2010}
J.~M. H.~Zhao, Q.~Ran, L.~Tan, Generalized prolate spheroidal wave functions
  associated with linear canonical transform, IEEE Transactions on Signal
  Processing 58~(6) (2010) 3032--3041.
\newblock \href {http://dx.doi.org/10.1109/TSP.2010.2044609}
  {\path{doi:10.1109/TSP.2010.2044609}}.

\bibitem{ZWSW2014}
D.~S. H.~Zhao, R.~Wang, D.~Wu, Maximally concentrated sequences in both time
  and linear canonical transform domains, Signal, Image and Video Processing
  8~(5) (2014) 819--829.
\newblock \href {http://dx.doi.org/10.1007/s11760-012-0309-1}
  {\path{doi:10.1007/s11760-012-0309-1}}.

\bibitem{BHHA2008}
A.~H. M.~Bahria, E.~M.~Hitzer, R.~Ashinob, An uncertainty principle for
  quaternion fourier transform, Computers and Mathematics with Applications
  56~(9) (2008) 2398--2410.
\newblock \href {http://dx.doi.org/10.1016/j.camwa.2008.05.032}
  {\path{doi:10.1016/j.camwa.2008.05.032}}.

\bibitem{E1993}
Quaternion-fourier transfotms for analysis of two-dimensional linear
  time-invariant partial differential systems, Proceeding of the 32nd
  Conference on Decision and Control, San Antonio, Texas.

\bibitem{H2007}
E.~M. Hitzer, Quaternion fourier transform on quaternion fields and
  generalizations, Advances in Applied Clifford Algebras 17~(3) (2007)
  497--517.
\newblock \href {http://dx.doi.org/10.1007/s00006-007-0037-8}
  {\path{doi:10.1007/s00006-007-0037-8}}.

\bibitem{HM2008}
E.~M. Hitzer, B.~Mawardi, Clifford fourier transform on multivector fields and
  uncertainty principles for dimensions $n=2(mod\ 4)$ and $n=3(mod\ 4)$
  advances in applied clifford algebras, Advances in Applied Clifford Algebras
  18~(3-4) (2008) 715--736.
\newblock \href {http://dx.doi.org/10.1007/s00006-008-0098-3}
  {\path{doi:10.1007/s00006-008-0098-3}}.

\bibitem{EBW1987}
Principles of nuclear magnetic resonance in one and two dimensions, Oxford
  University Press.

\bibitem{S1979}
A.~Sudbery, Quaternionic analysis, Mathematical Proceedings of the Cambridge
  Philosophical Society 85~(2) (1979) 199--225.
\newblock \href {http://dx.doi.org/http://dx.doi.org/10.1017/S0305004100055638}
  {\path{doi:http://dx.doi.org/10.1017/S0305004100055638}}.

\bibitem{BDS1982}
Clifford analysis, London: Pitman Research Notes in Mathematics.

\bibitem{D1975}
Interpolation and approximation, Blaisdell Publishing Company Press, New York.

\bibitem{GS1989}
Quaternionic analysis and elliptic boundary value problems, Blaisdell
  Publishing Company Press, New York.

\bibitem{CKL2015}
K.~K. L.~Chen, M.~Liu, Pitt's inequatlity and the uncertainty principle
  associated with the quaternion fourier transform, Mathematical Analysis and
  Applications 423~(1) (2015) 681--700.
\newblock \href {http://dx.doi.org/10.1016/j.jmaa.2014.10.003}
  {\path{doi:10.1016/j.jmaa.2014.10.003}}.

\bibitem{K1992}
Integral equations, Clarendon Press/Oxford University Press.

\bibitem{M2012}
The classical theory of integral equations a concise treatment, New York:
  Birkh.

\bibitem{MRR2013}
Convolution therorems for quaternion fourier transform: properties and
  applications, Abstract and Applied Analysis 2013.

\bibitem{DS1989}
D.~L. Donoho, P.~B. Stark, Uncertainty principles and signal recovery, SIAM
  Journal on Applied Mathematics 49~(3) (1989) 906--931.
\newblock \href {http://dx.doi.org/10.1137/0149053}
  {\path{doi:10.1137/0149053}}.
\end{thebibliography}
\end{document}